\documentclass[reqno,11pt]{amsart}
\usepackage{amsmath,amssymb,amsfonts,amsbsy,setspace,amsthm,graphicx,setspace,color}
\usepackage[a4paper,margin=3truecm]{geometry}

\usepackage{xcolor}
\usepackage{hyperref}
\usepackage{enumerate}
\usepackage{mathabx}
\usepackage{soul}
\usepackage[normalem]{ulem}
\newcommand{\stkout}[1]{\ifmmode\text{\sout{\ensuremath{#1}}}\else\sout{#1}\fi}

\theoremstyle{plain}
\newtheorem{theorem}{Theorem}[section]

\newtheorem{proposition}[theorem]{Proposition}
\newtheorem{lemma}[theorem]{Lemma}
\newtheorem{corollary}[theorem]{Corollary}
\theoremstyle{definition}
\newtheorem{definition}[theorem]{Definition}
\newtheorem{remark}[theorem]{Remark}
\newtheorem{example}[theorem]{Example}

\numberwithin{equation}{section}

\newcommand{\Prob}{\mathcal{P}}
\newcommand{\R}{\mathbb{R}}
\newcommand{\N}{\mathbb{N}}

\newcommand{\Rd}{\mathbb{R}^d}
\newcommand{\Leb}{\mathcal{L}}
\newcommand{\eps}{\varepsilon}
\newcommand{\kcal}{\mathcal{K}}
\newcommand{\ical}{\mathcal{I}}
\newcommand{\deps}{\mathcal{D}_\eps}
\newcommand{\depsN}{\mathcal{D}_\eps^N}
\newcommand{\peps}{\mathcal{C}_\eps}
\newcommand{\pepsN}{\mathcal{C}_\eps^N}
\newcommand{\uepsk}{u_{\eps_k}}
\newcommand{\vepsk}{v_{\eps_k}}

\newcommand{\xkik}{x_{k,i_k}}
\newcommand{\Qkik}{Q_{k,i_k}}
\newcommand{\xoneione}{x_{1,i_1}}

\newcommand{\xdots}{x_1,\dots,x_N}
\newcommand{\Xtimes}{X_1\times\cdots\times X_N}
\newcommand{\rhotimes}{\rho_1\otimes\cdots\otimes\rho_N}
\newcommand{\udots}{u_1,\dots, u_N}

\def\tran#1{ #1 ^{c,\varepsilon}}
\def\tranm#1{( #1 )^{c,\varepsilon}}
\newcommand{\LEnt}{L^{exp,\varepsilon}}
\DeclareMathOperator{\argmax}{argmax}
\DeclareMathOperator{\spt}{spt}

\definecolor{darkred}{rgb}{0.4,0,0}

\author{Camilla Brizzi}
\address{Department of Mathematics, Technical University of Munich \\ Boltzmannstr. 3,
	85748 Garching b. München \\ Germany}
\email{camilla.brizzi@tum.de}

\author{Luigi De Pascale}
\address{Dipartimento di Matematica e Informatica, Universit\`a di Firenze \\ Viale Morgagni 67/a, 50134 Firenze \\ Italy}
\email{luigi.depascale@unifi.it}
\urladdr{http://web.math.unifi.it/users/depascal/}

\author{Anna Kausamo}
\address{Dipartimento di Matematica e Informatica, Universit\`a di Firenze \\ Viale Morgagni 67/a, 50134 Firenze \\ Italy}
\email{akausamo@gmail.com}

\title[The variational limit of entropic regularizations]{A regularized transportation cost stemming from entropic approximation. }


\begin{document}
	\maketitle
	\begin{abstract}
		We study the entropic regularizations of optimal transport problems under suitable summability assumptions on the point-wise transport cost. These summability assumptions already appear in the literature. However, we show that the weakest compactness conditions that can be derived are already enough to obtain the convergence of the regularized functionals. This approach allows us to characterize the variational limit of the regularization even when it does not converge  to the original problem. The results apply also to problems with more than two marginals. 
	\end{abstract}	
	
	\section{Introduction} 
	Let $X$ and $Y$ be Polish spaces, $\mu \in \Prob (X) $ and $\nu \in \Prob (Y)$ be Borel probability measures, and let $c:X\times Y \to [0, +\infty]$ be a lower semi-continuous transport cost. Consider the functional
	\begin{equation}\label{hmkproblem}
		\peps (\gamma)= \int_{X\times Y} c(x,y) d \gamma(x,y) + \eps \mathcal{H} (\gamma|\mu\otimes \nu),
	\end{equation}
	where $\gamma\in\Pi(\mu,\nu):=\{\gamma \in \Prob(X\times Y) \, : \, \pi^1_\sharp \gamma = \mu \ \text{and} \ \pi^2_\sharp \gamma = \nu\}$ and 
	\begin{align}\label{entropy}
		\notag	&\mathcal{H} (\gamma|\mu\otimes \nu)=\\  &=\begin{cases}
			\int_{X\times Y}\left(\frac{d\gamma}{d(\mu\otimes \nu)}(x,y)\right)\log\left(\frac{d\gamma}{d(\mu\otimes \nu)}(x,y)\right)d\mu\otimes\nu(x,y)-1&\text{ if }\gamma\ll\mu\otimes\nu\\
			+\infty&\text{otherwise}
		\end{cases}
	\end{align}
	is the relative entropy functional of $\gamma$ with respect to $\mu\otimes\nu$. 
	The Entropic Optimal Transport Problem (or entropy-regularized optimal transport problem, and EOT in short) is the variational problem
	\begin{equation}\label{EOT}\tag{EOT}
		\min_{\gamma\in\Pi(\mu,\nu)} \peps(\gamma).
	\end{equation}
	
	Variational problems involving the relative entropy have a long history in mathematics and applications \cite{leonard2014survey,ChaChaLie1984CMP}. The Schr\"odinger problem is the king and paradigma of this kind of problems. 
	
	Here we look at the relative entropy as a regularizing term for the standard optimal transport problem.  The first to use the entropy in this fashion, and to remark that a fast algorithm was possible for this problem, seems to be \cite{Cut2013ANIPS}. The reader will find more on the computational aspects in \cite{CutPey2016,CutuPeyr2019FTML,BenCarCutNenPeySJSC2015}. $\Gamma$-convergence as $\eps\to 0$ of the problem \eqref{EOT} to the standard optimal transport problem 
	\[\min_{\gamma \in \Pi(\mu,\nu)}\int_{X\times Y} c(x,y)d\gamma(x,y),\]
	has been proved in \cite{CarDuvPeySch2017,LeoJFA2012} for the quadratic cost. Results about different type of convergence and for more general continuous costs can be found, for instance, in \cite{Nut2020LecNotes,CarPegTamCalVarPDE2023}. In the last years also an investigation on the accuracy of the convergence has flourished, see for instance  \cite{CarPegTamCalVarPDE2023} and the references therein, and \cite{NenPegCJM2023} for a multi-marginal result. 
	
	The problem \eqref{EOT} has a natural dual counterpart, given by %(see \eqref{dual} below) which is a maximization problem in $L^1_\mu \times L^1_\nu$. 
%	The dual formulation of the problem \eqref{EOT} is given by 
	\begin{equation}\label{dual}\tag{D}
		\max_{(u,v)\in L^1_\mu \times L^1_\nu }\deps(u,v),
	\end{equation} 
	where
	\begin{equation}\label{dualfunctional}
		\deps(u,v):= \int_X u(x) d\mu(x) +\int_Y v(y) d\nu(y) -\eps \int e^\frac{u(x)+v(y)-c(x,y)}{\eps} d\mu\otimes \nu(x,y) +\eps,
	\end{equation}
	for $u \in L^1_\mu$ and $v\in L^1_\nu$. Moreover
	$\deps (u,v)>-\infty$ if and only if $(u,v)\in \LEnt_\mu \times \LEnt_\nu$ (see Proposition \ref{domain}) where, given a measure $\lambda$ on a Polish space $Z$, 
	\begin{align*}
		&\LEnt_\lambda :=\{f \in L^1_\lambda \ : \ \int_Z e^\frac{f}{\eps}d\lambda <+\infty\}.
	\end{align*}
	Thus
	\begin{equation}
		\max_{(u,v)\in L^1_\mu \times L^1_\nu }\deps(u,v)=\max_{(u,v)\in \LEnt_\mu \times \LEnt_\nu }\deps(u,v).
	\end{equation}

	The existence of a maximizing pair has been investigated in several papers. For a continuous and bounded cost and also for more general entropy terms see \cite{DiMGer2020JSC,DiMGer2020OptimalTL}. There a maximizing sequence is proved to converge weak* in  $L^\infty$ while a more specific maximizing sequence given by the Sinkhorn algorithm (see \eqref{eq:sequencesinkhorn}) converges strong in $L^p$. A previous existence result in statistic mechanical setting can be found in \cite{ChaChaLie1984CMP}. Convergence in total variation of the Sinkhorn algorithm is studied in \cite{NutWie2023AoP}
	under an integrability assumption on the continuous cost. Convergence of the dual maximizers as $\eps \to 0$ can be found, for example, for continuous costs which satisfy an integrability condition in \cite{Nut2020LecNotes,NutWie2023AoP}.
Moreover, the continuity - or, more precisely, the upper semicontinuity - of the cost function $c$ is a sufficient condition for the maximizers of \eqref{dual} to converge to the Kantorovich potentials of the associated optimal transport problem (see \cite{NutWie2022PTRF}). However, it is important to stress the fact that, as observed in \cite{Nut2020LecNotes,NutWie2022PTRF} (see also Example \ref{ex: counterexample}),  if the cost $c$ is not continuous, the regularized problem might not converge to the unregularized one with the same cost function $c$.
	One of the main result of this paper is an analysis of the limit for a lower semicontinuous $c$ which satisfies some integral bounds (see assumption \eqref{ipotesi1} below). In particular we show that the limit potentials of a sequence of optimal potentials for \eqref{dual} are, indeed, the Kantorovich potentials for an optimal transport problem whose cost function $\tilde c$ is a regularization of  the original cost function $c$, which for every $(x,y)\in X\times Y$ is defined by
	\begin{equation}\label{eq:ctilde}
		\tilde{c}(x,y):= \sup\{ t \, : \, \exists r>0, \, \text{s.t.} \, \mu\otimes\nu (B_r(x,y)\cap \{c>t\})= \mu\otimes\nu (B_r((x,y)) \}.
		\end{equation}
	Even if apparently obscure, this regularization belongs to the same $L^1(\mu\otimes\nu)$ - class of $c$ and replace the lack of upper semicontinuity of $c$ with a simlar property taylored for the measure $\mu\otimes\nu$ (see Remark \ref{rmk:uppersemicont_ctilde}).

In this paper we assume that $\mu$, $\nu$, and $c$ satisfy 
	%\sout{the following assumptions}:
%	\begin{equation}\label{ipotesi2}\tag{$A_1$} 
%		\stkout{\int_{X\times Y} c(x,y) d\mu\otimes\nu(x,y)\leq \ical <+\infty}.
%	\end{equation}
%	\sout{and}
	the following assumption
	\begin{equation}\label{ipotesi1}\tag{$A$} 
		\sup_y \int_X c(x,y) d\mu(x),\ \sup_x \int_Y c(x,y) d\nu(y)\leq \kcal <+\infty.
	\end{equation}
%Under the assumption \eqref{ipotesi2} the existence of a solution for \eqref{EOT} is quite standard (see for instance \cite{ BorLewNuss1994JFA} and \cite{Nut2020LecNotes, leonard2014survey} for a detailed survey). Moreover, by the strict convexity of $\gamma\mapsto \peps(\gamma)$, such a solution is also unique. 
The assumption \eqref{ipotesi1} implies	\begin{equation*}%\label{ipotesi2}\tag{$A_1$} 
		\int_{X\times Y} c(x,y) d\mu\otimes\nu(x,y)\leq \ical <+\infty,
	\end{equation*}
	which is a common requirement to have existence of a solution (see for instance \cite{ BorLewNuss1994JFA} and \cite{Nut2020LecNotes, leonard2014survey} for a detailed survey). Moreover, by the strict convexity of $\gamma\mapsto \peps(\gamma)$, such a solution is also unique. 
%	While condition \eqref{ipotesi2} is natural, assumption \eqref{ipotesi1} is stronger. 
The reason for condition \eqref{ipotesi1}, which can be found also in \cite{Lel2024ESAIMCOCV}, is that we are interested in cost functions $c$ that can admit $+\infty$ as a value like, for instance, the important case of the Coulomb cost. The optimal transport problem with Coulomb cost, especially in its multi-marginal formulation, has been largely studied in the last years, as it appears when considering a system of electrons in the case where the repulsion between them largely dominates over their kinetic energies (see for instance \cite{ButtDePaGoGi2012PhRA,FriGerGor2022,CotFrieKluCPAM2013}).
In the case of Coulomb cost, in \cite{GerKauRaj2020}  and \cite{Lel2024ESAIMCOCV} has been showed respectively that the primal and dual regularized optimal transport problem converge, as $\eps\to 0$, to the unregularized one with the same cost. In the second paper more is said on the regularity of potentials.

Below we provide some examples of cost functions and marginals that satisfy these assumptions which, then, appear quite reasonable. 
	
	\begin{example}
		Let us give a few examples of $\mu,\nu$ and $c$ satisfying %assumptions \eqref{ipotesi2} and 
		assumption \eqref{ipotesi1} above. 
		\begin{enumerate}
			\item Let $X=Y=\Rd$, $d\ge 3$, $c(x,y)=\frac{1}{|x-y|^{d-2}}$ and $\mu=\nu\ll\Leb^d$ with  $d\mu(x)=d\nu(x)=\rho(x) dx$. %This example can be found in \cite[Remark~ 4]{Lel2024ESAIMCOCV}. 
%			\sout{Condition \eqref{ipotesi2} is clearly satisfied whenever $c\in L^1_{\mu\otimes \nu}(\Rd)$. In this case, by the Hardy-Littlewood Sobolev inequality (see for instance \cite{LieLos2001AMS}), it translates to asking that }$\stkout{\rho\in L^{\frac{2d}{d+2}+\eps}(\Rd)}$ \footnote{\sout{Take} $\stkout{\lambda= d-2}$ \sout{and} $\stkout{p=q}$ such that $\stkout{\frac2p+\frac{d-2}{d}=2}$}. 
			Condition \eqref{ipotesi1} is satisfied whenever $\rho\in L^{\frac{d}{2}+\eps}(\Rd)$, for some $\eps>0$. Indeed, if we fix $p=\frac d 2 +\eps=\frac{d}{2-\eta}$, with $\eta=\frac{2\eps}{d+\eps}$, and $q=\left(1-\frac 1p\right)^{-1}=\frac{d}{d-2+\eta}$, by the  H\"{o}lder inequality we get 
			\begin{equation*}
				\left(\int_{\R^d}\left(\frac{1}{|x-y|}\right)^{(d-2)q}dx\right)^{\frac 1q}\left(\int_{\Rd} \rho^p(x)dx\right)^{\frac 1 p}= \left|\left||x|^{-1} \right|\right|_{q(d-2)}^{\frac{1}{q}}  ||\rho||_{p},
			\end{equation*}
			where $q(d-2)<d$ guarantees that $\left|\left||x|^{-1} \right|\right|_{q(d-2)}$ is bounded. Notice that the first integral does not depend on $y$. The Gaussian density $\rho(x)=\frac{1}{\sqrt{2\pi}}e^{-\frac 12 \frac{(x-m)^2}{\sigma^2}}$, $m\in \R, \sigma^2 >0$, for instance, satisfies the required conditions.
%			\item The example above can be extended to the case of $c(x,y)=h(x-y)$, where $h\in L^r(\Rd)$ for some $r\ge 1$, and $d\mu(x)=f(x)dx$, $d\nu(y)=g(y)dy$. \sout{Then condition \eqref{ipotesi2} is satisfied  whenever} $\stkout{f\in L^p(\Rd)}$ and $\stkout{g\in L^q(\Rd)}$, with $\stkout{\frac 1p+ \frac 1q+\frac 1r=2}$,\sout{ thanks to the Young inequality. Indeed by the Young inequality we get}
%			\begin{equation*}
%			\stkout{	\int_{\Rd} \int_{\Rd} h(x-y)f(x)g(y)dxdy \le C_{p,q,r,d}||f||_p||g||_q||h||_r,}
%			\end{equation*}
%			\sout{for some} $\stkout{C_{p,q,r,d}\ge0}$. 
By the H\"{older} inequality as before, if $f,g\in L^{r'}(\Rd)$, with $r'=\left(1-\frac 1 r\right)^{-1}$  we obtain \eqref{ipotesi1}.  
		\end{enumerate}
	\end{example}
In this paper we focus on the dual problem \eqref{dual}.
	The reason to consider the functional \eqref{dualfunctional} is natural, since by using standard tools of convex analysis (see Proposition \ref{prop:ineq1Dual} in the appendix) one can see that
	\[\peps(\gamma)\ge \deps(u,v), \quad \text{for any }\gamma\in\Pi(\mu,\nu), u,v\in L^1_\mu\times L^1_\nu. \]
	However the equality 
	\[\min_{\gamma\in\Pi(\mu,\nu)}\peps(\gamma)= \max_{(u,v)\in \LEnt_\mu \times \LEnt_\nu }\deps(u,v)\] is more delicate. 
	One way to show it is to first prove a characterization of the unique minimizer $\gamma_\eps$ of \eqref{EOT} which says that $\frac{d\gamma_\eps}{d \mu\otimes \nu}(x,y) = e^{\frac{u_\eps(x)+v_\eps(y)-c(x,y)}{\eps}}$, for some $(u_\eps,v_\eps)\in L^1_\mu\times L^1_\nu$ and then to show that the pair $(u_\eps,v_\eps)$ is a maximizer for \eqref{dual}.  For the characterization of the unique optimal $\gamma_\eps$ see for instance \cite{BorLewNuss1994JFA,FolGan1997ANNProb}, for a more comprehensive discussion and the equivalence of the two problems we refer to \cite{Nut2020LecNotes}. Here, on the contrary, we provide a direct proof of existence of maximizers for \eqref{dual}. From this result the above mentioned characterization of $\gamma_\eps$ and the consequent equivalence of the two problems follows directly (see Corollary \ref{cor:minimizer}). The proof of the existence of maximizers is a by-product of the proof of the convergence of the Sinkhorn's algorithm (see Theorem \ref{thm:convergenceSinkhorn}) which is based on a compactness argument. This argument exploits the nice properties of the $(c,\eps)$-transforms (see \eqref{eq:trasfentrop}) which are not only uniformly bounded but also locally uniformly continuous (see Proposition \ref{prop:regolaritactransgen}) in a way that does not depend on $\eps$. A natural consequence is the convergence of the maximizers as $\eps\to 0$.  Section \ref{sec: lim eps} is dedicated to show that the limits of such maximizers are the Kantorovich potentials of the optimal transport problem associated with cost defined in \eqref{eq:ctilde}. An \textit{ad hoc} construction of a Vitali covering (see \ref{thm:vitalicovering}), suitable for general Polish spaces and inspired by the one in \cite{KaeRajSuo2012PAMS}, allows to show that $c$ and its regularization coincide $\mu\otimes\nu$-a.e.. Moreover the minimum values and, up to subsequences, the minimizers of the primal problem \eqref{EOT} converge, respectively, to the minimum value and the minimizers of the primal optimal transport problem with $\tilde c$ (see Theorem \ref{thm:convergenceprimal}).
	Finally we remark that, as pointed out in Section \ref{sec:MM}, the same bound on the $(c,\eps)$-transforms generalizes to the multi-marginal setting (see Proposition \ref{betterpotentials}). This allows for an extension of all our results to that case.
	
	\section{uniform bounds on the $(c,\eps)$-transforms and the existence of maximizers}\label{sec:preliminary}
	This section is devoted to the study of basic but interesting properties of the entropic transform of a function (see Proposition \ref{htrasform} below). We will see that those properties imply a nice bound on the entropic transforms of suitable potentials (see Proposition \ref{improve2}). This bound will be crucial in our proofs of the existence of maximizers (see Theorem \ref{existsmax}) for the dual functional and of the convergence of the Sinkhorn's algorithm (see Theorem \ref{thm:convergenceSinkhorn} in Section \ref{sec:regularityandSiknhorn}) as well as in the discussion on convergence  when $\eps$ is sent to $0$ (see Section \ref{sec: lim eps}).
	
	\begin{proposition}\label{htrasform}
		For every $u\in \LEnt_\mu$ and for every $v\in \LEnt_\nu$, $\deps(u,\cdot)$ and $\deps(\cdot, v)$ are uniquely maximized respectively by 
		\begin{equation}\label{eq:trasfentrop}
			\tran u(y):= -\eps \log \int_X e^{\frac{u(x)-c(x,y)}{\eps}}d\mu(x),
		\end{equation}
		and by
		\begin{equation}\label{eq:trasfentropv}
			\tran v(x):= -\eps \log \int_Y e^{\frac{v(y)-c(x,y)}{\eps}}d\nu(y),
		\end{equation} 
		which are called the entropic transforms or $(c,\eps)$-transforms of $u$ and $v$.
		
	\end{proposition}
	\begin{proof} See Proposition \ref{prop:aboututransf}.
	\end{proof}
	\begin{corollary}	\label{cor:increasingdualfunct}
		For every $u\in \LEnt_\mu$ and $v\in \LEnt_\nu$
		\begin{equation}\label{improve}
			\deps(u,v)\leq \deps (u,\tran{u}) \quad \text{and} \quad \deps(u,v)\leq \deps (\tran{v},v),
		\end{equation}
		and that 
		\begin{equation*}
			\deps(u,v)= \deps(u,\tran{u}) \quad \text{and} \quad \deps(u,v)= \deps (\tran{v},v),
		\end{equation*}
		if and only if $v=\tran{u}$ and $u=\tran{v}$.

	\end{corollary}  
	
	\begin{lemma}\label{tranupluscost} Let $u \in \LEnt$ and $a\in \R$, then $\tran{(u+a)}=\tran{u}-a$.
	\end{lemma}
	\begin{proof}
		\begin{align*}
			\tran{(u+a)}(y)=-\eps \log e^{\frac{a}{\eps}}\int_X e^{\frac{u(x)-c(x,y)}{\varepsilon}}d\mu(x)=\tran{u}(y)-a.
		\end{align*} 
	\end{proof}
	
	\begin{lemma}\label{simpledual} Let $u\in L_\mu ^{exp,\eps}$ and $v\in L_\nu ^{exp,\eps}$. For any pair of the type $(u,\tran{u})$ or $(\tran{v},v)$ 
		\[\deps (u,\tran{u})= \int_X u d\mu+\int_Y \tran{u} d\nu \quad \text{and} \quad \deps (\tran{v},v)= \int_X \tran{v} d\mu+\int_Y v d\nu.\]
	\end{lemma}
	\begin{proof} 
		It is enough to substitute in the definition of $\deps(u,v)$ (see equation \ref{dualfunctional}), the explicit formula of $\tran{u}$ when $v=\tran{u}$ and of  $\tran{v}$ when $u=\tran{v}$.
		
	\end{proof}
	\begin{corollary}\label{lem:positiveintegraltransform} 
		Let $u\in L_\mu ^{exp,\eps}$ and $v\in L_\nu ^{exp,\eps}$ such that $\deps(u,v)=\int u d\mu +\int v d\nu$. Then 
		\[\int \tran{v} d\mu\ge \int u d\mu  \quad \text{and} \quad \int \tran{u} d\nu\ge \int v d\nu.  \]
	\end{corollary}
	\begin{proof}
		By Lemma \ref{simpledual} and equation \eqref{improve} we have
		\[\int_X u d\mu+\int_Y \tran{u} d\nu=\deps (u,\tran{u})\ge \deps (u,v)= \int_X u d\mu+\int_Y v d\nu. \]
	\end{proof}
	\begin{lemma}\label{boundabove} Let $u\in \LEnt_\mu$. If $\int u d\mu \ge 0$, then $\tran{u} \leq \kcal$, where $\mathcal{K}$ is the same as in \eqref{ipotesi1}.
	\end{lemma}
	\begin{proof} For every $y\in Y$ we have by Jensen's equality
		\begin{align*}
			\tran{u}(y)=-\eps \log \int_X e^{\frac{u(x)-c(x,y)}{\varepsilon}}d\mu(x)  \le -\int_X u(x)d\mu(x)+\int_Xc(x,y)d\mu(x)\le \kcal.
		\end{align*}
	\end{proof}

	\begin{lemma}\label{cbelow}  Let $u\in \LEnt_\mu$. If $u \leq A$, for some $A\in \R$, then $\tran{u} \geq -A$.
	\end{lemma}
	
	\begin{proof}
		Since $c(x,y)\ge 0$ we get
		\[\tran{u} (y)=-\eps \log \int_X e^{\frac{u(x)-c(x,y)}{\varepsilon}}d\mu(x)\ \ge -\eps \log \int_X e^{\frac{A}{\varepsilon}}d\mu(x) =-A.\]

	\end{proof}

	\begin{proposition}\label{improve2} If $(u,v)$ is such that $\deps(u,v)\geq 0$, then there exist $(\tilde{u},\tilde{v})$ satisfying 
		\begin{enumerate}
			\item $\deps(\tilde{u},\tilde{v}) \geq \deps(u,v) $,
			\item $-\kcal \leq \tilde{u}, \tilde{v} \leq \kcal$,
			\item \label{prop3} $\int \tilde{u} d\mu \geq 0,\ \int \tilde{v} d\nu \geq 0$. 
		\end{enumerate}
		
	\end{proposition}
	\begin{proof} Starting from the pair $(u,v)$ consider
		\[a_0=\frac{\deps(u,\tran{u})}{2}-\int u d\mu\] 
		Notice that by equation \eqref{improve} $\deps(u,\tran{u}) \geq \deps(u,v)\geq 0$.  
		We first define 
		$(\overline u, \overline v):=(u+a_0, \tran{u}-a_0)$. Observe that $(\overline u, \overline v)=(\overline u, \tran{\overline u})$ by Lemma \ref{tranupluscost} and that, since $\deps$ is invariant under $(u,v)\mapsto (u+a,v-a)$,
		\[\deps(\overline{u},\overline v)=\deps (u,\tran{u})\ge \deps (u,v)\ge 0,\] 
		where the first inequality comes by \eqref{improve}.
		Moreover we have
		\[\int \overline u d\mu = \int u d\mu +  \frac{\deps(u,\tran{u})}{2} -\int u d\mu= \frac{\deps(u,\tran{u})}{2}\geq 0.\] 
		By Lemma \ref{boundabove} it follows that $\overline{v}\leq \kcal$ and by Lemma \ref{simpledual} that 
		\[\int \overline v d\nu = \int \tran{u} d\nu -  \frac{\deps(u,\tran{u})}{2} +\int u d\mu=
		\frac{\deps(u,\tran{u})}{2} \geq 0.
		\]
		If we define $\tilde{u}:=\tran{\overline v}$ and $\tilde{v}:=\tran{\tilde{u}}$, then  by Lemma \ref{cbelow}, $\tilde{u}\geq -\kcal$.
		and by Lemma \ref{boundabove}, $\tilde u \le \mathcal {K}$ Again by Lemma \ref{cbelow}, $\tilde u \le \mathcal {K}$.
		Moreover, by applying twice Corollary \ref{lem:positiveintegraltransform}, we see that $(\tilde u, \tilde v)$ satisfies Property \ref{prop3}. Thus, by Lemma \ref{boundabove}, $\tilde v \le \mathcal K$. Finally, again by \eqref{improve}, $\deps(\tilde u, \tilde v)\ge \deps(\overline u, \overline v)\ge 0$.

	\end{proof}
	
	\begin{corollary}\label{preparazione} There exists $(u_0,v_0)$ such that 
		\begin{enumerate}
			\item$\deps(u_0,v_0) \geq 0$,
			\item $u_0,v_0 \leq \kcal$,
			\item $\int u_0 d\mu \geq 0,\ \int v_0 d\nu \geq 0$,
			\item $v_0=\tran{u_0}$.
		\end{enumerate}
		
	\end{corollary}
	\begin{proof}
		Apply the proposition above starting with $u\equiv 0$ and $v \equiv 0$. Then $u_0=\tilde u$ and $v_0=\tilde v$.
	\end{proof}
	\begin{theorem}\label{existsmax}For every $\eps>0$ there exists a unique\footnote{Up to summing and substracting the same constant, i.e. $(u_\eps, v_\eps)$ is a maximizer iff $(u_\eps+a, v_\eps-a)$ is a maximizer for some $a\in\R$. } pair  $(u_\eps,v_\eps)$ maximizing 
		$\deps$. Moreover $-\mathcal{K}\le u_\eps,v_\eps\le \mathcal{K}$, where $\mathcal{K}$ is given by \eqref{ipotesi1}, and $u_\eps=\tran{v_\eps}$, $v_\eps=\tran{u_\eps}$.
	\end{theorem}
	\begin{proof} Let $(u_n,v_n)_n$ be a maximizing sequence. Without loss of generality \footnote{indeed $\deps (0,0)\ge 0$, so the supremum is nonnegative.} we may assume that $\deps(u_n,v_n)\ge 0$ for every $n$. By Proposition \ref{improve2} (i.e. passing to $(\tilde{u}_n, \tilde{v}_n)$) we may assume that $-\kcal \leq u_n, v_n\leq \kcal$.
		By Banach-Alaoglu theorem, up to subsequences
		\[(u_n,v_n)\rightharpoonup (u,v)\] 
		in $L^2_\mu \times L^2_\nu$.
		Since $((x,y), (\xi_1,\xi_2)) \mapsto \eps e^\frac{-c(x,y)}{\eps} e^\frac{\xi_1+\xi_2}{\eps}- \xi_1 -\xi_2$ is convex and lower semi-continuous with respect to $(\xi_1, \xi_2)$ for every fixed $(x,y)$ and  is bounded from below by $c(x,y) +\eps (1-\log \eps)$, which is in $L^1_{\mu \otimes \nu}$, then by Theorem \ref{semicontinuity} $\deps$ is weakly upper semi-continuous in $L^2_\mu \times L^2_\nu$ and thus $(u_\eps,v_\eps)=(u,v)$ is a maximizer. The uniqueness follows from the strict concavity. The fact that $u_\eps=\tran{v_\eps}$ and $v_\eps=\tran{u_\eps}$ follows directly by Lemma \ref{htrasform}.
	\end{proof}
	A direct consequence of the above result is the following
	\begin{corollary}\label{cor:minimizer}
		Let $(u_\eps,v_\eps)$ the pair of maximizers given by Theorem \ref{existsmax}, then the coupling $\bar\gamma_\eps\in\Pi(\mu,\nu)$, given by \begin{equation}\label{eq: minimizer}
			d\bar\gamma_\eps:=e^{\frac{u_\eps\oplus v_\eps - c}{\eps} }d\mu\otimes\nu
		\end{equation} is the unique minimizer of the problem \eqref{EOT}.
		Moreover 
		\begin{equation*}
			\min_{\Pi(\mu,\nu)} \peps (\gamma)=\max_{L^1_\mu\times L^1_\nu}\deps(u,v) 
		\end{equation*}
	\end{corollary}
	\begin{proof}
		Thanks to Proposition \ref{prop:ineq1Dual}, it is enough to show that $\bar\gamma_\eps\in \Pi(\mu,\nu)$ and check that $\peps(\bar\gamma_\eps)=\deps(u_\eps,v_\eps)$. 
	\end{proof}
	\section{Uniform continuity of  the $(c,\eps)$-transform and convergence of the Sinkhorn's algorithm} \label{sec:regularityandSiknhorn}
	\subsection{Regularity of the $(c,\eps)$-transform}
	In this section we report and show some results on the regularity of the $c,\eps$-transform. We start with a well-known fact stated in Proposition \ref{prop:reularity1} and due to S. Di Marino and A. Gerolin (see \cite{DiMGer2020JSC}). Even though it is not directly used for the results later on, it is interesting in itself and we rewrite it here for completeness. The key result for proving the convergence of Sinkhorn's algorithm is given by Proposition \ref{prop:regolaritactransgen}, which is the local uniform continuity of the $(c,\eps)$-transform. The same result can be found in  \cite{DiMGer2020JSC} as well, but proved under the stronger assumption of a bounded cost function. 
	Finally, Proposition \ref{prop:lipschitz} shows a stronger regularity in the special case of Coulomb cost. 	
	
	\begin{proposition}\label{prop:reularity1}
		If $c:X\times Y \to \R$ is uniformly continuous then for any $u\in L_\eps^{exp}$, $\tran{u}$ is a uniformly continuous function with the same modulus of continuity of $c$. In particular, if $c$ is Lipschitz, $u^{c,\eps}$ has the same Lipschitz constant of $c$.  
	\end{proposition}
	
	\begin{proof}
		Let $u\in L_{\eps}^{exp}$ and $y_1,y_2\in Y$. We may assume without loss of generality that $\tran{u}(y_1)\ge \tran{u}(y_2)$. Then 
		\begin{align*}
			&|\tran{u}(y_1)-\tran{u}(y_2)|=\\&=\eps \log\left(\int_{X}e^{\frac{u(x)-c(x,y_2)}{\varepsilon}}d\mu(x)\right)-\eps\log\left(\int_{X}e^{\frac{u(x)-c(x,y_1)}{\varepsilon}}d\mu(x)\right) \\
			&=\eps \log\left(\int_{X}e^{\frac{u(x)-c(x,y_1)+c(x,y_1)-c(x,y_2)}{\varepsilon}}d\mu(x)\right)-\eps\log\left(\int_{X}e^{\frac{u(x)-c(x,y_1)}{\varepsilon}}d\mu(x)\right) \\
			&\le \eps \log\left(e^{\frac{\omega(d(y_1,y_2))}{\eps}}\int_{X}e^{\frac{u(x)-c(x,y_1)}{\varepsilon}}d\mu(x)\right)-\eps\log\left(\int_{X}e^{\frac{u(x)-c(x,y_1)}{\varepsilon}}d\mu(x)\right) \\
			&= \omega (d(y_1,y_2)),
		\end{align*}
		where $\omega$ is the modulus of continuity of $c$.
	\end{proof}
	\begin{proposition}\label{prop:lipschitz}
		Let $X=Y$ be a Polish space, $\alpha>0$ and  $c(x,y):=\frac{1}{d(x,y)^\alpha}$. If \space $||u||_{\infty}<M$, then $\tran{u}$ is Lipschitz with the Lipschitz constant that depends only on $M$, $\alpha$ and $\eps$. In particular this result holds for $X=Y=\Rd$ and the Coulomb cost $c(x,y):=\frac{1}{|x-y|^{d-2}}$.
	\end{proposition}
	\begin{proof}
		Let  $-M \leq u \leq M$, 
		\[\tran{u}(y)=-\varepsilon\log \int_X e^{\frac{u(x)-c(x,y)}{\varepsilon}}d\mu(x).\]
		Using Jensen's inequality for the convex function $t\mapsto e^t$ we get
		\begin{align*}
			\int_X e^{\frac{u(x)-c(x,y)}{\varepsilon}}d\mu(x)&\ge \int_X e^{\frac{-M-c(x,y)}{\varepsilon}}d\mu(x)=e^{-\frac{M}{\eps}}\int_X e^{\frac{-c(x,y)}{\varepsilon}}d\mu(x)\\
			&\ge e^{-\frac{M}{\varepsilon}}e^{\frac{-\int c d\mu}{\varepsilon}}\ge e^{-\frac{M+\mathcal{K}}{\varepsilon}},
		\end{align*}
		where $\kcal$ is the constant of \eqref{ipotesi1}.\\
		Since the function $t\mapsto -\log t$ is Lipschitz on $[e^{-\frac{M+\mathcal{K}}{\varepsilon}},\infty)$ with a Lipschitz constant $e^{\frac{M+\mathcal{K}}{\varepsilon}}$, we only need to prove that the function 
		\[y\mapsto \int_Xe^{\frac{u(x)-c(x,y)}{\varepsilon}}d\mu(x)\]
		is Lipschitz with a Lipschitz constant that does not depend on the choice of $u$ with the bounds above.
		We observe that the function $f(t)= e^{-\frac{1}{\eps t^\alpha}}$ is Lipschitz on $(0,\infty)$ with Lipschitz constant
		\[ L(\eps)=(\frac{\eps}{\alpha})^\frac{1}{\alpha} (\alpha+1)^\frac{\alpha+1}{\alpha} e^{- \frac{\alpha+1}{\alpha}}. \]
		Therefore,  for $y_1,y_2\in \Rd$, we have 
		\begin{align*}
			& \left\lvert\int_{X}e^{\frac{u(x)-c(x,y_1)}{\varepsilon}}d\mu(x)-\int_{X}e^{\frac{u(x)-c(x,y_2)}{\varepsilon}}d\mu(x)\right\rvert \le \\
			&e^{\frac{M}{\varepsilon}}\int_{X}\left\lvert e^{-\frac{1}{\varepsilon d(x,y_1)^{d-2}}}-e^{-\frac{1}{\varepsilon d(x,y_2)^{d-2}}}\right\rvert d\mu(x)\le L(\eps) e^{\frac{M}{\varepsilon}} d(y_1,y_2).
		\end{align*}
		Putting everything together we obtain
		\begin{multline*}
			\left\lvert-\eps \log\int_{\Rd}e^{\frac{u(x)-c(x,y_1)}{\varepsilon}}d\mu(x)+\eps\log\int_{\Rd}e^{\frac{u(x)-c(x,y_2)}{\varepsilon}}d\mu(x)\right\rvert\le\\ \eps e^{\frac{2M+\mathcal{K}}{\eps}}L(\eps)d(y_1,y_2).
		\end{multline*}
		
	\end{proof}
	
	And we remark that in this case the Lipschitz constant blows up with $\eps$.
	
	\begin{proposition}\label{prop:regolaritactransgen}
		Let $X,Y$ be Polish spaces. If $||u||_{\infty}\le A$ for some $A\ge 0$, then for every $\delta>0$, there exists a set $N^\delta\subset Y$, such that $\nu( N^\delta)<\delta$, $K^\delta:=Y\setminus N^\delta$ is compact and 
		\begin{equation}
			|\tran{u}(y_1)-\tran{u}(y_2)|\le \omega_\delta(d(y_1,y_2)), \quad \text{for every} \ y_1,y_2\in K^\delta,
		\end{equation}
		where $\omega_\delta$ is a modulus of continuity.
	\end{proposition}
	\begin{proof}
		\underline{Step 1}: Since $c\in L^1_{\mu\otimes\nu}$, by Lusin's theorem, for every $\sigma>0$ there exists $M^\sigma\subset X\times Y$ such that $\mu\otimes\nu( M^\sigma)<\sigma$ and, fixed $C^\sigma:=X\times Y\setminus M^\sigma$, $c_{|C^\sigma}$ is continuous. Since $\mu,\nu$ are inner regular \footnote{indeed $\mu,\nu$ are Borel probability measures}, $C^\sigma$ can be chosen to be a compact set and thus $c$ is uniformly continuous on $C^\sigma$.  We will denote by $\omega_\sigma$ the modulus of continuity of $c$ on $C^\sigma$.  Let $M^\sigma_y:=\{x\in X \, : \, (x,y)\in M^\sigma\}\subset X$. 
		We define the set of \textit{bad} $y$'s as $N^\sigma:=\{ y \in Y \, : \, \mu(M^\sigma_y)\ge \sqrt{\sigma}  \}$. We observe that 
		\begin{equation*}
			\int_{N^\sigma}\mu(M^\sigma_y)d\nu(y)\le\mu\otimes\nu(M^\sigma)<\sigma
		\end{equation*}implies $\nu(N^\sigma)<\sqrt\sigma$.
		Let us consider $y_1,y_2\not\in N^\sigma$ and denote $X^\ast=X\setminus (M^\sigma_{y_1}\cup M^\sigma_{y_2})$. Then for every $x\in X^\ast$, $|c(x,y_1)-c(x,y_2)|\le \omega_\sigma (d(y_1,y_2))$. We assume without loss of generality that $\tran{u}(y_1)\ge \tran{u}(y_2)$. Then 
		\begin{align*}
			&|\tran{u}(y_1)-\tran{u}(y_2)|=\\&=-\eps\log\left(\int_{X}e^{\frac{u(x)-c(x,y_1)}{\varepsilon}}d\mu(x)\right)+\eps \log\left(\int_{X}e^{\frac{u(x)-c(x,y_2)}{\varepsilon}}d\mu(x)\right) \\
			&=\eps \log\left(\frac{\int_{X}e^{\frac{u(x)-c(x,y_1)+c(x,y_1)-c(x,y_2)}{\varepsilon}}d\mu(x)}{\int_{X}e^{\frac{u(x)-c(x,y_1)}{\varepsilon}}d\mu(x)}\right) \\
			&=\eps \log\left(\frac{\int_{X^\ast}e^{\frac{u(x)-c(x,y_1)+c(x,y_1)-c(x,y_2)}{\varepsilon}}d\mu(x)+\int_{M^\sigma_{y_1}\cup M^\sigma_{y_2}}e^{\frac{u(x)-c(x,y_2)}{\varepsilon}}d\mu(x)}{\int_{X}e^{\frac{u(x)-c(x,y_1)}{\varepsilon}}d\mu(x)}\right)
			\\&\le\eps \log\left(e^{\frac{\omega_\sigma(d(y_1,y_2))}{\eps}}\frac{\int_{X^\ast}e^{\frac{u(x)-c(x,y_1)}{\varepsilon}}d\mu(x)}{\int_{X}e^{\frac{u(x)-c(x,y_1)}{\varepsilon}}d\mu(x)}+\frac{\int_{M^\sigma_{y_1}\cup M^\sigma_{y_2}}e^{\frac{u(x)-c(x,y_2)}{\varepsilon}}d\mu(x)}{\int_{X}e^{\frac{u(x)-c(x,y_1)}{\varepsilon}}d\mu(x)}\right)
			\\&\le \eps \log\left(e^{\frac{\omega_\sigma(d(y_1,y_2))}{\eps}}+\frac{\int_{M^\sigma_{y_1}\cup M^\sigma_{y_2}}e^{\frac{u(x)-c(x,y_2)}{\varepsilon}}d\mu(x)}{\int_{X}e^{\frac{u(x)-c(x,y_1)}{\varepsilon}}d\mu(x)}\right).
		\end{align*}
		Now we notice that, since $||u||_\infty < A$ and $c\ge 0$,
		\begin{align*}
			\int_{M^\sigma_{y_1}\cup M^\sigma_{y_2}}e^{\frac{u(x)-c(x,y_2)}{\varepsilon}}d\mu(x)\le \mu(M^\sigma_{y_1}\cup M^\sigma_{y_2})e^{\frac{A}{\eps}},
		\end{align*}  and 
		\begin{align*}
			\int_{X}e^{\frac{u(x)-c(x,y_1)}{\varepsilon}}d\mu(x)\ge e^{\int_{X}{\frac{u(x)-c(x,y_1)}{\varepsilon}}d\mu(x)}\ge e^{\frac{-A-\mathcal{K}}{\eps}},
		\end{align*}
		where $\mathcal{K}$ is the constant in assumption \eqref{ipotesi1} on $c$. We also observe that,  since $y_1,y_2\not\in N^\sigma$, $\mu(M^\sigma_{y_1}\cup M^\sigma_{y_2})<2\sqrt{\sigma}$, by the definition of $N^\sigma$.
		So, putting everything together we get 
		\begin{align*}
			|\tran{u}(y_1)-\tran{u}(y_2)|&\le \eps\log\left(e^\frac{\omega_\sigma(d(y_1,y_2))}{\eps}+ \mu(M^\sigma_{y_1}\cup M^\sigma_{y_2})e^{\frac{2A+\mathcal{K}}{\eps}}\right)\\&\le \omega_\sigma(d(y_1,y_2))+\eps2\sqrt{\sigma}e^{\frac{2A+\mathcal{K}}{\eps}},
		\end{align*} 
		where in the last inequality we have used the fact that, if $a,b\ge 0$, then $e^a+b\le e^{a+b}$.
		
		\underline{Step 2}: We now fix $\delta>0$ and we construct a sequence $\sigma_k$ such that $\sum_{k=1}^\infty \nu( N^{\sigma_k})\le \delta $. We define $K^{\sigma_k}:= Y\setminus N^{\sigma_k}$, $N^\delta=\bigcup_{k}N^{\sigma_k}$ and $K^\delta =\bigcup_k K^{\sigma_k}$. Then $\nu(N^\delta)=\nu(\bigcup N^{\sigma_k})<\delta$, $\nu(K^\delta)>1-\delta$, and by the inner regularity of $\nu$ can be chosen \footnote{by possibly replacing $K^\delta$ with a compact set included in $K^\delta$ and with measure $\nu$ greater than $1-\delta$.} to be compact. Now, for every $y_1,y_2\in K^\delta$ we have 
		\begin{equation*}
			|\tran{u}(y_1)-\tran{u}(y_2)|\le \omega_{\sigma_k}(d(y_1, y_2)) + \beta_k, \quad \text{for every } k\in\N,
		\end{equation*}where $\beta_k:=\eps2\sqrt{\sigma_k}e^{\frac{A+\mathcal{K}}{\eps}}$. Therefore, by defining  $\omega_\delta(t):=\inf_{k}\{\omega_k(t) +\beta_{\sigma_k} \}$, one obtains
		\begin{equation*}
			|\tran{u}(y_1)-\tran{u}(y_2)|\le \omega_\delta (d(y_1,y_2)), \quad \text{for every} \ y_1,y_2\in K^\sigma.
		\end{equation*}
		We conclude by showing that $\omega_\delta$ is a modulus of continuity. First of all we notice that as an infimum of non-decreasing functions it is non-decreasing. Moreover, for any given $\tilde\eps$, let $k$ be such that $\beta_k\le \frac{\tilde \eps}{2}$ and $t$ such that $\omega_{\sigma_k}(t) <\frac{\tilde \eps}{2}$, then $\omega_\delta(t)< \tilde \eps$. 
	\end{proof}

	\subsection{Sinkhorn algorithm}
	
	Let $(u_0,v_0)$ be given by Corollary \ref{preparazione}. For $n\ge 0$,  we define the sequence 
	\begin{equation}\label{eq:sequencesinkhorn}
		\begin{cases}
			u_{n+1}:= \tran{v_n},\\
			v_{n+1}:= \tran{u_{n+1}}.
		\end{cases}
	\end{equation}
	\begin{theorem}\label{thm:convergenceSinkhorn}
		Let  $c:X\times Y\to [0,+\infty]$ be a cost function that satisfies the assumption \eqref{ipotesi1}. Then the sequence \eqref{eq:sequencesinkhorn} converges uniformly on compact sets to the unique maximizer  $(\bar u,\bar v)$ of $\deps$.
	\end{theorem}
	\begin{proof}
		Thanks to the property of $(u_0,v_0)$, applying Lemma \ref{boundabove} and Lemma \ref{cbelow} to each $(u_n,v_n)$ we know that $||u_n||_\infty, ||v_n||_\infty \le \mathcal{K} $. For every $\delta>0$, let us consider the compact $K_1^\delta\subset X$ and $K_2^\delta \subset Y$ given by Proposition \ref{prop:regolaritactransgen}.  Then the sequence $\{(u_n,v_n)\}$ is uniformly continuous on the compact set $K^\delta:=K_1^\delta\times K_2^\delta$, for every $\delta>0$. 
		Let us consider $(K^{\delta_m})_m$, with $\delta_m\to 0$, as $m\to\infty$. Then $\mu(\bigcup_m K_1^{\delta_m})=\nu(\bigcup_m K_2^{\delta_m})=1$ by construction. By the Arzel\`a-Ascoli Theorem and the usual diagonal argument there exists a subsequence $(u^k_{n_k},v^k_{n_k})$ that converges uniformly on $K^{\delta_m}$, for every $m$, to a pair  $(\bar u, \bar v)$, with $||\bar u||_\infty, ||\bar v||_\infty \le \mathcal{K} $ and $(u^k_{n_k},v^k_{n_k})\to (\bar u, \bar v)$ $\mu\otimes \nu$-a.e. Therefore for any fixed $y$ in $\bigcup_m K_2^{\delta_m}$
		\begin{align}
			\label{eq:dominatedconvergence1}&e^{\frac{u_{n_k}(x)-c(x,y)}{\eps}}\to e^{\frac{\bar u(x)-c(x,y)}{\eps}}, \quad \text{for }\mu\text{- a.e.} \ x. 
		\end{align}
		Since $\left|e^{\frac{u_{n_k}(x)+v_{n_k}(y)-c(x,y)}{\eps}}\right|\le e^{\frac{2\mathcal{K}}{\eps}}$ and $\left|e^{\frac{u_{n_k}(x)-c(x,y)}{\eps}}\right|\le e^{\frac{2\mathcal{K}}{\eps}}$, for any fixed $y$, by the dominated convergence theorem, we have that 
		$\deps (u^k_{n_k},v^k_{n_k})\to \deps(\bar u,\bar v) $ 
		and that 
		\begin{align}\label{utov}
			\bar v(y)&=\lim_{k\to\infty}v_{n_k}(y)=\lim_{k\to\infty}-\eps\log\int e^{\frac{u_{n_k}(x)-c(x,y)}{\eps}}d\mu(x)\\
			&=-\eps\log\int e^{\frac{\bar u(x)-c(x,y)}{\eps}}d\mu(x)=\tran{\bar u}(y).
		\end{align}
		Since we would like to show that also $\bar u(x)= \tran{\bar v}(x)$, we consider the subsequence $(u_{n_{k}-1},v_{n_{k}-1})$ which converges, possibly up to subsequences, uniformly on compact sets and to some $(u',v')$. 
		By the definition and fundamental properties of $\tran{u}$ and $\tran{v}$, we know that (see equation \eqref{improve})
		\begin{equation}\label{eq:increasingdepsthm}
			\cdots\ge\deps (u_{n+1},v_{n+1})\ge \deps(u_n,v_n)\ge \deps (u_{n-1},v_{n-1})\ge \cdots,
		\end{equation} 	
		therefore 
		\begin{equation*}
			\deps(\bar u, \bar v)=\lim_{k\to \infty} \deps(u_{n_{k}},v_{n_{k}})=\lim_{k\to \infty} \deps(u_{n_{k}-1},v_{n_{k}-1})=\deps(u',v'),
		\end{equation*}
		We now observe that 
		\begin{align*}
			\bar u(x)&=\lim_{k\to \infty} u_{n_k}(x)=-\lim_{k\to \infty}\eps\log\int e^{\frac{v_{n_k-1}(y)-c(x,y)}{\eps}}d\nu(y)\\
			&=-\eps\log\int e^{\frac{v'(y)-c(x,y)}{\eps}}d\nu(y)=\tran{{v'}}.
		\end{align*}
		This implies 
		\begin{equation*}
			\deps(\bar u, \bar v)=\deps(u',v')\le \deps (\bar u,v'),
		\end{equation*}
		which, since $\bar v=\tran{\bar u}$, implies $\bar v=v'$ and thus $\bar u=\tran{\bar v}$. 
	\end{proof}
	
	\section{The limit as $\eps \to 0$}\label{sec: lim eps}
	In this section we investigate the limit $\eps\to0$ both for the primal and dual problems. 
	\subsection{The dual problem}
	Let us define the set 
	\[\mathcal{B}':=\{ u\in L^1_\mu, v\in L^1_\nu, \ u(x)+v(y) \leq c(x,y), \, \mu\otimes \nu \text{-a.e.} (x,y)\}.\]
	We prove the following convergence result as $\eps\to 0$ for the dual problem \eqref{dual}.
	\begin{proposition}\label{prop-sd}
		Let $(u_\eps,v_\eps)$, be the maximizer of $\mathcal{D}_\eps$. There exists a subsequence $(u_{\eps_k},v_{\eps_k})$ such that $u_{\eps_k}\stackrel{*}{\rightharpoonup}u^*$ in $L^\infty_\mu$  and $v_{\eps_k}\stackrel{*}{\rightharpoonup} v^*$  in $L^\infty_\nu$ and 
		\begin{equation}
			\label{eq:limitdual}
			\lim_{k\to\infty}\mathcal{D}_{\eps_k}(\uepsk,\vepsk)=\int_{X}u^* d\mu+\int_{Y}v^*d\nu.
		\end{equation}  
		Moreover the pair $(u^*,v^*)\in L^1_\mu \times L^1_\nu$ maximizes 
		\begin{equation}\label{dualestrano}
			\max_{(u,v)\in\mathcal{B}'}\int u d\mu +\int v d\nu.
		\end{equation}
	\end{proposition}
	\begin{proof}
		By Theorem \ref{existsmax} we know that $-\kcal\le u_\eps,v_\eps\leq \kcal$ for every $\eps>0$. The $L^\infty$-weak* convergence, up to a subsequences, of $u_\eps$ and $v_\eps$ respectively to $u^*$ and $v^*$ follows directly by the Banach-Alaoglu theorem.
		This implies that $\uepsk(x)+\vepsk (y)\stackrel{\ast}{\rightharpoonup} u^\ast (x)+v^\ast(y)$ in $L^\infty_{\mu\otimes\nu}$ as $k\to \infty$, and therefore
		\begin{align*}
			\lim_{k\to\infty}\mathcal{D}_{\eps_k}(\uepsk,\vepsk)=\lim_{k\to\infty}\int_{X}\uepsk d\mu+\int_{Y}\vepsk d\nu=\int_{X}u^* d\mu+\int_{Y}v^* d\nu,
		\end{align*}
		where the first equality is due Lemma \ref{simpledual}, since $u_\eps=\tran{v_\eps}$ and $v_\eps=\tran{u_\eps}$.
		Let us prove that $u^\ast(x)+v^\ast(y)\le c(x,y)$, for $\mu\otimes\nu$-a.e $(x,y)$.
		Assume, by contradiction, that there exists $\alpha>0$ and $B_\alpha\subset X\times Y$ such that $\mu\otimes\nu(B_\alpha)>0$ and 
		\begin{equation*}
			u^\ast(x)+v^\ast(y)-c(x,y)>\alpha \quad \text{on} \ B_\alpha.
		\end{equation*}
		Since the indicator function of $B_\alpha$ is in $L^1_{\mu\otimes\nu}$, for $k$ large enough 
		\begin{equation*}
			\int_{B_\alpha}\uepsk(x)+\vepsk(y)-c(x,y) d {\mu\otimes\nu}(x,y) >\frac{\alpha}{2} \mu\otimes\nu(B_\alpha)=:\tilde \alpha
		\end{equation*}
		and, therefore, there exists $\lambda>0$ such that for every $k$ large enough the set $B_k:=\{(x,y)\in B_\alpha \, : \,  \uepsk(x)+\vepsk(y)-c(x,y)>\frac{\tilde\alpha}{2} \}$ satisfies $\mu\otimes\nu (B_k) >\lambda$. 
		Indeed, if $\mu\otimes\nu (B_k)\to 0$, then for $k$ large enough
		\begin{align*}
			&\int_{B_\alpha}\uepsk(x)+\vepsk(y)-c(x,y) d\mu\otimes \nu(x,y)\\& = \int_{B_k}\uepsk(x)+\vepsk(y)-c(x,y)d\mu\otimes \nu(x,y)+\int_{B_\alpha\setminus B_k}\uepsk(x)+\vepsk(y)-c(x,y)d\mu\otimes \nu(x,y)\\& \le 2\kcal\mu\otimes\nu(B_k)+ \mu\otimes\nu(B_\alpha)\frac{\tilde\alpha}{2}<\tilde\alpha.
		\end{align*}
		Therefore one gets that 
		\begin{align*}
			\mathcal{D}_{\eps_k}(\uepsk,\vepsk) &\le 2\mathcal{K} -\eps_k\int_{B_k}e^{\frac{\uepsk(x)+\vepsk(y)-c(x,y)}{\eps_k}}d {\mu\otimes\nu (x,y)} +\eps_k\\& \le  2\mathcal{K}   -\eps_k e^{\frac{\tilde \alpha}{2 \eps_k}} \mu\otimes\nu(B_k) +\eps_k,
		\end{align*}
		which tends to $-\infty$ as $k$ tends to $+\infty$, contradicting \eqref{eq:limitdual}.\\
		We conclude by showing the maximality of $(u^\ast,v^\ast)$. Let $(u,v)$ be such that $u+v\le c$, $\mu\otimes\nu$-a.e. Then 
		\begin{align*}
			&\int_{X}u^*(x) d\mu(x)+\int_{Y}v^*(y) d\nu(y)=\lim_{k\to\infty}\int_{X}\uepsk(x) d\mu(x)+\int_{Y}\vepsk(y) d\nu(y)\\& \ge \lim_{k\to\infty}\int_{X}u(x) d\mu(x)+\int_{Y}v(y )d\nu(y)-\eps_k\int_{X\times Y}e^{\frac{u(x)+v(y)- c(x,y)}{\eps_k}}d\mu\otimes\nu(x,y)  +\eps_k\\&=\int_{X}u(x) d\mu(x)+\int_{Y}v(y) d\nu(y).
		\end{align*}
	\end{proof}
	\begin{remark}
		As shown for instance in \cite[Lemma 2.6]{NutWie2022PTRF}, it is not hard to see that if $c$ is upper semicontinuous (and thus continuous) then
		\begin{equation*}
			\mathcal{G}'=\mathcal{B}',
		\end{equation*}
		where
		\begin{align*}
			\mathcal{G}':&=\{ u\in L^1_\mu, v\in L^1_\nu, \ u(x)+v(y) \leq c(x,y), \,  \mu\text{-a.e.} \ x,  \, \nu\text{-a.e.} \ y\}.
		\end{align*}
		
		In particular, 
		the pair $(u^\star,v^\star)$ given by Proposition \ref{prop-sd} satisfies
		\begin{align*}
			\int u^\star d\mu + \int v^\star d\nu = \max_{(u,v)\in\mathcal{G}'}\int u d\mu + \int v d\nu = \min_{\gamma \in \Pi(\mu,\nu)}\int c(x,y)d\gamma,
		\end{align*}
		by duality in optimal transport.
	\end{remark}
	However, as shown in the example below, if $c$ is not upper semicontinuous, it might happen $\mathcal{G}'\subsetneq \mathcal{B}'$ and 
	\begin{equation*}
		\max_{(u,v)\in\mathcal{G}'}\int u d\mu + \int v d\nu < \max_{(u,v)\in\mathcal{B}'}\int u d\mu + \int v d\nu.
	\end{equation*}
	\begin{example}\label{ex: counterexample}
		Let $X=Y=\R$, $\mu=\nu=\Leb^1_{|[0,1]}$ and $c:\R\times\R\to[0,+\infty]$ be defined by
		\begin{equation*}
			c(x,y):=\begin{cases}
				\frac{1}{2} \quad \text{if} \ 0\le x\le 1 \ \text{and} \ y<x;\\
				1 \quad \text{if} \ 0\le x\le 1 \ \text{and} \ y>x;\\
				0 \quad \text{if} \  0\le x\le 1 \  \text{and} \ y=x.
			\end{cases}
		\end{equation*}
		In this case the pair $(u,v)\equiv(\frac14,\frac14)\in\mathcal{B}'\setminus\mathcal{G}'$. Indeed if $(u,v)\in \mathcal{G}'$ then $u(x)+v(x)\le c(x,x)=0$. In particular \[\max_{(u,v)\in\mathcal{G}'}\int u d\mu + \int v d\nu=0<\frac12 \le \max_{(u,v)\in\mathcal{B}'}\int u d\mu + \int v d\nu. \]
	\end{example}
%	Let  $\tilde c:X\times Y\to \R \cup \{+\infty\}$ is the regularization of $c$ defined in \eqref{effetilda}, i.e, for any $(\bar x, \bar y)\in X\times Y$,
%	\begin{align*}
%		&\tilde{c}(\bar x,\bar y):=\\ &\sup\{ t \, : \, \forall \eps, \exists r>0, \, \text{s.t.} \, \mu\otimes\nu (B_r((\bar x,\bar y)\cap \{t-\eps < c\})= \mu\otimes\nu (B_r((\bar x,\bar y)) \}.
%	\end{align*}
Let  $\tilde c:X\times Y\to \R \cup \{+\infty\}$ is the regularization of $c$ defined in \eqref{eq:ctilde}, i.e, for any $(\bar x, \bar y)\in X\times Y$,
\begin{equation*}
	\tilde{c}(\bar x,\bar y):=\sup\{ t \, : \, \exists r>0, \, \text{s.t.} \, \mu\otimes\nu (B_r(\bar x,\bar y)\cap \{c>t\})= \mu\otimes\nu (B_r(\bar x,\bar y)) \}.
\end{equation*}
All the interesting properties of this regularization can be found in Theorem \ref{propertiestilde}.
	\begin{proposition}  \label{prop2-ds}
		\begin{equation}\label{dualitastrana}
			\max_{(u,v)\in\mathcal{B}'}\int u d\mu +\int v d\nu\leq\min_{\gamma \in \Pi(\mu,\nu)} \int \tilde{c} (x,y) d \gamma=:\min_{\gamma \in \Pi(\mu,\nu)}\tilde{\mathcal{C}}(\gamma).
		\end{equation}
	\end{proposition}
	\begin{proof}
		Let us define 
		\begin{equation*}
			\mathcal{B}:=\{u\in C_b(X), v\in C_b(Y), \ u(x)+v(y) \leq c(x,y), \, \mu \otimes \nu\text{-a.e.}\}.
		\end{equation*}
		Then, by the density of the continuous and bounded functions in $L^1$,
		\begin{equation*}
			\max_{(u,v)\in\mathcal{B}'}\int u d\mu + \int v d \nu = \sup_{(u,v)\in\mathcal{B}}\int u d\mu + \int v d \nu.
		\end{equation*}
		
		Let us also define 
		\begin{align*}
			& \tilde {\mathcal{G}}:= \{ u\in C_b(X), v\in C_b(Y), \ u(x)+v(y) \leq \tilde c(x,y), \,  \mu\text{-a.e.} \ x,  \, \nu\text{-a.e.} \ y\}.
		\end{align*}
		In order to prove \eqref{dualitastrana}, it is sufficient to prove that $\mathcal{B}\subseteq \tilde {\mathcal{G}}$. Indeed, duality in optimal transport ensures that
		\begin{equation*}
			\min_{\gamma \in \Pi(\mu,\nu)}\tilde {\mathcal{C}}(\gamma)=\sup_{(u,v)\in \tilde {\mathcal{G} }}  \int u d\mu +\int v d\nu.
		\end{equation*}
		Let $(u,v)\in \mathcal{B}$ and let $(\bar x, \bar y)\in\spt \mu\otimes \nu=\spt \mu \times \spt \nu$. We want to show that $u(\bar x )+ v(\bar y)\le \tilde c(\bar x, \bar y)$. 
		By the definition of $\tilde c$, for every $\alpha>0$ and  for every $r>0$,
		\[\mu\otimes\nu(\{(x,y) \, : \, \{ \tilde c(\bar x, \bar y)+\alpha >c(x,y)\} \cap B_r(\bar x)\times B_r(\bar y) \})>0\]
		and, in particular,  
		\[\mu\otimes\nu(\{(x,y) \, : \, \{\tilde c(\bar x, \bar y)+\alpha >c(x,y)\} \cap B_r(\bar x)\times B_r(\bar y)\cap F \})>0, \]
		where $F:=\{(x,y)\, : \, u(x)+v(y)\le c(x,y)\}$, which has $\mu\otimes\nu$ full measure by the  definition of $\mathcal{B}$. This allows us to define a sequence $(x_n,y_n)$, such that  $(x_n,y_n)\in \{(x,y) \, : \, \tilde c(\bar x, \bar y)+\alpha >c(x,y)\cap B_{r_n}(\bar x)\times B_{r_n}(\bar y)\cap F \}$, for every $n\in\N$, and $r_n\to 0$ as $n\to \infty$. Then we have 
		\begin{equation*}
			u(\bar x)+ v(\bar y)= \lim_{n\to \infty} u(x_n)+ v(y_n)\le \liminf_{n\to \infty}c(x_n,y_n)\le c(\bar x, \bar y)+ \alpha,
		\end{equation*}
		where the first equality holds thanks to the continuity of $u$ and $v$. We conclude by the arbitrariness of $\alpha$.
	\end{proof}
	\begin{remark} \label{rmk:uppersemicont_ctilde}
	Thanks to Theorem \ref{propertiestilde}, we know that $\tilde c= c$, $\mu\otimes\nu $- a.e.. In other words, $\tilde c$ is a representative of the $L^1(\mu\otimes\nu)$ - class of $c$. From the proof of Proposition \ref{prop2-ds}, one can see that $\tilde c$ is constructed so that, if there is a point $(x,y)$
	where the inequality $ \phi(x) + \psi (y) > \tilde c (x,y)$ holds, then the same inequality is satisfied in a set of $\mu \otimes \nu$-positive measure in a ball around
	$(x,y)$. 
	Since $\phi(x)+\psi(y)\le \tilde c(x,y)$, $\mu\otimes \nu$ - a.e., this leads
	to a contradiction and therefore $\phi(x)+\psi(y)\le \tilde c(x,y)$
	everywhere in the support of $\mu\otimes\nu$.
This happens whenever the cost function is upper semicontinuous. So this choice of $\tilde c$ in some sense replace upper semicontinuity with a weaker property which depends on the probability measure $\mu\otimes\nu$. 
	\end{remark}
	\subsection{The primal problem}
	%annagamma
	\begin{theorem}\label{thm: liminfprimal}
		Let $\gamma_\eps,\gamma \in \Pi (\mu,\nu)$ be such that $\gamma_\eps \stackrel{*}{\rightharpoonup} \gamma$. Then 
		\[\liminf_{\eps\to 0} \peps (\gamma_\eps)\geq  \int \tilde{c}(x,y) d\gamma.\]
	\end{theorem}
	
	\begin{proof} Let $\gamma_\eps$ and $\gamma$ be as in the assumptions. Without loss of generality we may assume that $\gamma_\eps\ll \mu \otimes \nu$ and this, by Theorem \ref{propertiestilde}, implies 
		the second equality in the chain below.
		\begin{eqnarray*}
			\liminf_{\eps\to 0} \peps (\gamma_\eps)&=&	\liminf_{\eps\to 0}   \int c(x,y) d \gamma_\eps+\eps \mathcal{H} (\gamma_\eps |\mu\otimes \nu ) \\
			&=&  \liminf_{\eps\to 0} \int \tilde c(x,y) d \gamma_\eps+\eps \mathcal{H} (\gamma_\eps |\mu\otimes \nu )\\
			& \geq& 	\liminf_{\eps\to 0}   \int \tilde{c}(x,y) d\gamma_\eps \\
			& \geq & \int \tilde{c}(x,y) d\gamma.
		\end{eqnarray*}
		The last inequality is, again, consequence of Theorem \ref{propertiestilde} which states the lower semicontinuity of $\tilde c$. 
	\end{proof}
	\begin{corollary}
		\begin{equation*}
			\max_{(u,v)\in\mathcal{B}'}\int u d\mu +\int v d\nu = \min_{\gamma \in \Pi(\mu,\nu)} \int \tilde{c} (x,y) d \gamma.
		\end{equation*}
	\end{corollary}
	\begin{proof}
		Let $(u_{\eps_k},v_{\eps_k})$ be a sequence given by Proposition \ref{prop-sd} and $\bar\gamma_{\eps_k}$ be the unique minimizer of $\mathcal{C}_{\eps_k}$ (see Corollary \ref{cor:minimizer}). Then 
		\begin{align*}
			&\min_{\gamma \in \Pi(\mu,\nu)} \int \tilde{c}  d\gamma\le\liminf_{k\to\infty}\mathcal{C}_{\eps_k}(\bar\gamma_{\eps_k})\\&=\lim_{k\to\infty} \deps(u_{\eps_k},v_{\eps_k})=\max_{(u,v)\in\mathcal{B}'}\int u d\mu +\int v d\nu\leq\min_{\gamma \in \Pi(\mu,\nu)} \int \tilde{c}  d\gamma,
		\end{align*}
		where the first inequality is due to Theorem \ref{thm: liminfprimal}, the first equality is due to Corollary \ref{cor:minimizer}, the second equality to Proposition \ref{prop-sd} and the last inequality to Proposition \ref{prop2-ds}.
	\end{proof}
	\begin{theorem}\label{thm:convergenceprimal} Let $(\bar\gamma_\eps)_\eps$ be the sequence of minimizers of $\peps$ (see Corollary \ref{cor:minimizer}). Then there exists a subsequence $(\bar\gamma_{\eps_k})_k$ and $\gamma_0\in \Pi(\mu,\nu)$ such that $\bar\gamma_{\eps_k} \stackrel{*}{\rightharpoonup} \gamma_0$ and 
		\[
		\ \mathcal{C}_{\eps_k} (\bar\gamma_{\eps_k}) \to \tilde{\mathcal{C}}(\gamma_0).\]
		In particular $\gamma_0$ is a minimizer for $\tilde{\mathcal{C}}$.
	\end{theorem}
	\begin{proof} Let $(u_{\eps_k},v_{\eps_k})$ be the subsequence  given by Proposition \ref{prop-sd}, i.e. s.t. 
		$u_{\eps_k}\stackrel{*}{\rightharpoonup} u^*$ and 	$v_{\eps_k}\stackrel{*}{\rightharpoonup} v^*$.
		Then taking $\bar\gamma_{\eps_k}$ given by \eqref{eq: minimizer}, up to subsequences, we have that
		\begin{align*}
			 \int \tilde{c} d \gamma_0&\le\liminf_{k\to \infty} \int c d \gamma_{\eps_k}+ \eps_k \mathcal{H} (\gamma_{\eps_k}| \mu \otimes\nu)\\&=	\lim_{k\to \infty}  \int u_{\eps_k }d\mu + \int v_{\eps_k} d \nu= \int u^* d\mu + \int v^* d \nu \le \min_{\gamma\in\pi(\mu,\nu)}\int\tilde cd\gamma.
		\end{align*}
		where the first inequality is due to Theorem \ref{thm: liminfprimal}, the first equality to Corollary \ref{cor:minimizer} and the last inequality follows from Proposition \ref{prop2-ds} above.
	\end{proof}
	
	\section{The multi-marginal case} \label{sec:MM}
	In this section we show that all the results proved in the previous sections hold also in the multi-marginal case. Let $N\ge 2$. Let $X_1,\dots, X_N$ be Polish spaces and let $c:X_1\times\cdots\times X_N \to [0, +\infty]$ be a lower semi-continuous transport cost. Given $\rho_1,\dots,\rho_N$ such that $\rho_i\in\Prob(X_i)$ for every $i=1,\dots,N$, we consider the functional 
	\begin{equation}\label{MMhmkproblem}
		\peps^N (\gamma):= \int_{\Xtimes} c(\xdots) d \gamma(\xdots) + \eps \mathcal{H} (\gamma|\rhotimes),
	\end{equation}
	where $\gamma\in\Pi(\rho_1,\dots, \rho_N):=\{\gamma \in \Prob(X_1\times\cdots\times X_N ) \, : \, \pi^i_\sharp \gamma = \rho_i \ \text{for every} \ i=1,\dots, N\}$ and $\mathcal{H}(\gamma|\rho_1\otimes\cdots\otimes \rho_N)$ is the relative entropy of $\gamma$ with respect to $\rho_1\otimes\cdots\otimes \rho_N$, as defined in \eqref{entropy}. The multi-marginal Entropic Optimal Transport problem then reads
	\begin{equation}\label{MM-EOT}\tag{MM-EOT}
		\min_{\gamma \in \Pi(\rho_1,\dots, \rho_N)} \peps^N (\gamma).
	\end{equation}
	
	As mentioned in the introduction one special case of interest is when the marginals are $\rho_1,\ldots,\rho_N:={\bf \rho} \in\mathcal{P}(\R^{d})$ (in the physical case, usually, $d=3$), and the cost is the Coulomb cost
	\[c(x_1,\ldots,x_N)=\sum_{1\le i<j\le N}\frac{1}{|x_i-x_j|^{d-2}}.\]
	
	We denote by  $\rho^N=\bigotimes_{i=1}^N\rho_i$. As in the two-marginal case, we assume that  
	\begin{equation}\label{ipotesi2MM} \tag{$A^{M}$}
		\sup_{x_i}\int_{\bigtimes\limits_{j\neq i}X_j}c(\xdots)d\bigotimes\limits_{j\neq i} \rho_j(x_1,\dots,x_{i-1},x_{i+1},\dots,x_N)\le \mathcal{K}<+\infty,
	\end{equation}
	for every $1,\dots,N$. This clearly implies that
		\begin{equation*}\label{ipotesi1MM}
		\int_{\Xtimes}c(\xdots)d\rho^N(\xdots)<+\infty.
	\end{equation*}
	The dual functional associated to $\pepsN$ then reads
	\[\depsN(u_1,\ldots,u_N)=\sum_{i=1}^N\int_{X_i}u_i(x_i)d\rho_i(x_i)-\varepsilon\int_{X_1\times\cdots\times X_N}e^{\frac{\sum_{i=1}^Nu_i(x_i)-c(x_1,\ldots,x_N)}{\varepsilon}}d\rho^N+\varepsilon.\]
	As in the two-marginal case, $\depsN(u_1,\ldots,u_N)>-\infty$ if and only if $u_i\in \LEnt_{\rho_i}$, for every $i$, where $\LEnt_{\rho_i}$is the set defined in \eqref{Lexp}. The dual problem associated to the multi-marginal EOT problem is 
	\begin{equation}\label{MM dual}\tag{MM-D}
		\max_{L^1_{\rho_1}\times \cdots \times L^1_{\rho_N}}\depsN(u_1,\ldots,u_N).
	\end{equation} 
	From the definition of $\depsN$ it follows immediately that 
	\begin{remark}\label{sum}
		If $\{a_i\}_{i=1}^N$ are constants such that $\sum_{i=1}^Na_i=0$, then 
		\[\depsN(u_1+a_1,u_2+a_2,\ldots,u_N+a_N)=\depsN(u_1,\ldots, u_N).\]
	\end{remark}
	
	Analogously to the two-marginal case, the following fact holds.
	
	\begin{proposition}\label{basicproperties}
		For $u_1,\ldots,u_{N-1}$, with $u_i\in\LEnt_{\rho_i}$ the functional \[v\mapsto \depsN(u_1,\dots, u_{N-1}, v)\] admits a unique maximizer which is given by
		{\small	\[\tranm{u_1,\ldots,u_{N-1}}(x_N)=-\varepsilon \log\int_{X_1\times\cdots\times X_{N-1}}e^{\frac{\sum_{i=1}^{N-1}u_i(x_i)-c(x_1,\ldots,x_N)}{\varepsilon}}d\rho^{N-1}(x_1,\ldots,x_{N-1}),\]}
		and it is called the entropic transform or $(c,\eps)$-transform of $u_1,\dots, u_{N-1}$. The same holds for any $i\in \{1,\dots,N\}$.
	\end{proposition}
	\begin{proof}
		See Proposition \ref{prop:aboututransf} in the Appendix. It extends naturally to the multi-marginal case.
	\end{proof}
	\begin{corollary}
		For every $u_1,\dots u_N$, with $u_i\in\LEnt_{\rho_i}$
		\begin{equation*}
			\depsN(u_1,\dots,u_{N-1}, \tranm{u_1,\dots,u_{N-1}})\ge \depsN(\udots).
		\end{equation*}
		The same holds for any $i=1,\dots,N$
	\end{corollary}	
	The following are the same preliminary results on the $(c,\eps)$-transform that we have proved in Section \ref{sec:preliminary} for the two-marginal case. In particular Lemma \ref{tranupluscost}, Lemma \ref{simpledual} and Corollary \ref{lem:positiveintegraltransform} hold and can be proved in the same way; here we collect all the properties in a unique statement. 
	\\ \\  For the sake of simplicity we state the following properties for $\tranm{u_1,\dots,u_{N-1}}$. Everything holds, analogously, for $\tranm{u_1,\dots,u_{i-1},u_{i+1},\dots,u_N}$, for any $i\in \{1,\dots, N\}$.  
	\begin{lemma}\label{propMMtran}
		Let $\udots$, $u_i\in \LEnt_{\rho_i}$, then
		\begin{enumerate}
			\item Given $a_1,\ldots,a_{N-1}\in\R$ we have
			\[\tranm{u_1+a_1,u_2+a_2,\ldots,u_{N-1}+a_{N-1}}=\tranm{u_1,\ldots,u_{N-1}}-\sum_{i=1}^{N-1}a_i.\]
			\item  If $u_N=\tranm{u_1,\ldots,u_{N-1}}$ then one has
			\[\depsN(u_1,\ldots,u_{N-1},u_N)=\sum_{i=1}^N\int_{X_i}u_i(x_i) d\rho_i(x_i).\]
			\item If  $u_N=\tranm{u_1,\ldots,u_{N-1}}$, then for any $v$ such that \[\depsN(u_1,\ldots,u_{N-1},v)=\sum_{i=1}^{N-1}\int_{X_i}u_i(x_i) d\rho_i(x_i)+\int_{X_N} v(x_N) d\rho(x_N),\]
			\begin{equation*}
				\int u_N d\rho_N \ge \int v d\rho_N,
			\end{equation*}
		holds.
		\end{enumerate}
	\end{lemma}

	Also the bounds on the $(c,\eps)$-transforms  provided by Lemma \ref{boundabove}, Lemma \ref{cbelow} hold.
	\begin{lemma}\label{boundsMM}
		Let $\udots$, $u_i\in \LEnt_{\rho_i}$, then
		\begin{enumerate}
			\item If $\sum_{i=1}^{N-1}\int_{\R^d}u_i(x_i)d\rho_i(x_i)\ge 0$, then 
			\[\tranm{u_1,\ldots,u_{N-1}}\le \mathcal{K},\]
			where $\mathcal{K}$ is the constant in  \eqref{ipotesi2MM}.
			\item If  $\sum_{i=1}^{N-1}u_i(x_i)\le M$, for some $M\in\R$, then 
			\[\tranm{u_1,\ldots,u_{N-1}}\ge -M.\]
		\end{enumerate}
	\end{lemma}

	Proposition \ref{betterpotentials} below provides, for the multi-marginal case, the same bounds as Proposition \ref{improve2}. This is crucial for the bounds of the maximizing sequences and of the sequence generated by the Sinkhorn's algorithm. In this case, the proof requires some additional effort to iterate in the right way the strategy.
	
	\begin{proposition}\label{betterpotentials}
		Let $u_1,\ldots,u_N$, with $u_i\in\LEnt_{\rho_i}$, such that $\depsN(u_1,\ldots,u_N)\ge0$ and $u_N=\tranm{u_1,\ldots,u_{N-1}}$. Then there exist functions $\tilde u_1,\ldots,\tilde u_N$, with $\tilde u_i\in\LEnt_{\rho_i}$, such that
		\begin{enumerate}
			\item $\depsN(\tilde u_1,\ldots,\tilde u_N)\ge 0$:
			\item $\int \tilde u_i d\rho_i\ge 0\text{ for all }i=1,\ldots,N$:
			\item  $-(N-1)\mathcal{K}\le \tilde u_i\le \mathcal{K}$.
		\end{enumerate}

	\end{proposition}
	
	\begin{proof}
		Let $I_1:=\depsN(u_1,\ldots,u_N)$.
		Consider $a_1=\frac{I_1}{2}-\int u_1 d\rho_1$ and set $\bar u_1:=u_1+a_1$, then
		\[\int \bar u_1 d\rho_1=\int (u_1+a_1) d\rho_1=\frac{I_1}{2}\ge 0.\]
		We observe that with this choice of $a_1$, since $\depsN(\udots)=\sum_{i=1}^N\int u_i d\rho_i$, we have 
		\[\frac{I_1}{2}=\int (u_1+a_1) d\rho_1=\sum_{i=2}^N\int u_i d\rho_i-a_1\ge 0.\]
		Now we set
		\begin{align*}
			&I_2:=\frac{I_1}{2}=\sum_{i=2}^N\int u_i d\rho_i-a_1=\sum_{i=2}^N\int(u_i-\frac{a_1}{N-1})d\rho_i, \\
			&a_2:=\frac{I_2}{2}-\int (u_2-\frac{a_1}{N-1})d\rho_2 \quad \text{and}\\
			& \bar{u}_2 :=u_2-\frac{a_1}{N-1}+a_2
		\end{align*}
		Then $\int \bar u_2 d\rho_2 =\frac{I_2}{2}\ge 0.$
		In general, we set for $k\in \{2,\ldots,N-1\}$
		\begin{align*}
			&I_k:=\frac{I_{k-1}}{2}=\sum_{i=k}^N\int u_i -\sum_{j=1}^{k-1}\frac{a_j}{N-j}d\rho_i \\
			&a_k:=\frac{I_k}{2}-\int (u_k-\sum_{j=1}^{k-1}\frac{a_j}{N-j})d\rho_k \quad \text{and}\\
			& \bar{u}_k :=u_k-\sum_{j=1}^{k-1}\frac{a_j}{N-j}+a_k.
		\end{align*}
		Then 	$\int\bar{u}_k d\rho_k = \frac{I_k}{2}\ge 0$.
		Finally we set 
		\[\bar u_N := u_N -\sum_{j=1}^{N-1}\frac{a_j}{N-j}\]
		and we observe that 
		\begin{equation*}
			\int\bar{u}_N d\rho_N=\frac{I_{N-1}}{2}\ge 0.
		\end{equation*}
		
		Now $\depsN(u_1,\ldots,u_N)=\depsN(\bar u_1,\ldots,\bar u_N)$ because the constants we have added to the functions $u_i$ sum to zero. \\
		We replace $\bar u_1$ with $\tilde u_1:=\tranm{\bar u_2,\dots, \bar u_{N}}$. Then $\tilde u_1\le \mathcal K$, by Lemma \ref{boundsMM}, and $\int \tilde u_1 d\rho_1 \ge 0$ by Lemma \ref{propMMtran}. Let $\tilde u_i:= \tranm{\tilde u_1,\dots,\tilde u_{i-1}, \bar u_{i+1}, \dots \bar u_{N}}$, for every $i=2,\dots, N$. Then $\tilde u_i\le \mathcal K$, by Lemma \ref{boundsMM}, and $\int \tilde u_i d\rho_i \ge 0$ by Lemma \ref{propMMtran}. By iterating this procedure and replacing again $\tilde u_i$ with the $c,\eps$-transform of the others, which we will still call $\tilde u_i$ for simplicity, thanks to Lemma \ref{boundsMM}, we also obtain $\tilde u_i \ge -(N-1)\mathcal K$ and we conclude the proof.		
	\end{proof}
	
	\begin{corollary}\label{goodSequence}
		There exist $( u^0_1,\ldots, u^0_N)$ such that 
		\begin{enumerate}
			\item $\depsN( u^0_1,\ldots,\tilde u^0_N)\ge 0$;
			\item $\int u^0_i\rho_i\ge 0$, for every $i=1,\dots,N$:
			\item $-(N-1)\mathcal{K}\le u^0_i\le \mathcal{K}$.
		\end{enumerate} 
	\end{corollary}
	\begin{proof}
		Apply Proposition \ref{betterpotentials} to  $u_i\equiv 0$, for $i=1,\ldots,N-1$, and $u_N=\tranm{u_1,\dots,u_{N-1}}$. 
		Indeed 
		\begin{align*}
			\depsN(u_1,\ldots,u_N)&\ge\depsN(u_1,\ldots, u_{N-1},0)\\&=-\varepsilon\int e^{-\frac{c(\xdots)}{\varepsilon}} d\rho^N+\varepsilon \ge 0,
		\end{align*}
		since $c\ge 0$.
	\end{proof}
	As in the two-marginal case, thanks to the bounds given by Proposition \ref{betterpotentials}, the existence of maximizers follows:
	\begin{theorem} \label{th: maximizer MM} For every $\eps>0$ there exists a unique \footnote{Up to constants that sum to zero, i.e. $(u_1^\eps,\ldots,u_N^\eps)$ is a maximizer iff $(u_1^\eps+a_1,\ldots,u_N^\eps+a_N)$, with $\sum_{i=1}^N a_i=0,$ is a maximizer.} $N$-tuple $(u_1^\eps,\ldots,u_N^\eps)$ maximizing the problem \eqref{MM dual}. Moreover the plan $\bar\gamma_\eps\in \Pi(\rho_1,\dots,\rho_N)$ defined by $d\bar\gamma_\eps=e^{\frac{\bigoplus_{i=1}^N u^\eps_i-c}{\eps}}d\bigotimes_{i=1}^N\rho_i$ is the unique minimizer of \eqref{MM-EOT}.
	\end{theorem}
	\begin{proof}
		The proof is analogous to that of Theorem \ref{existsmax}: take a maximizing sequence and improve it by using Proposition \ref{betterpotentials}. Note that the semicontinuity theorem \ref{semicontinuity} trivially extends to the multi-maginal case. 
	\end{proof}
	As for the two-marginal cases, Proposition \ref{betterpotentials} and Corollary \ref{goodSequence} guarantee the existence of a uniformly bounded maximizing sequence. This bound does not depend on $\eps$ and it allows for the existence of the limit for $\eps\to 0$. The same discussion made in Section \ref{sec: lim eps} for the two marginal case holds also in this section. We summarize all those results in the following theorem.
	\begin{theorem}
		For every $\eps>0$, let $(u_1^\eps,\ldots,u_N^\eps)$ the N-tuple given by Theorem \ref{th: maximizer MM}. Then there exists a subsequence $((u_1^{\eps_k},\ldots,u_N^{\eps_k}))$ and a $N$-tuple $(u_1^\star,\dots, u_N^\star)\in L^1_{\rho_1}\times \cdots\times L^1_{\rho_N}$, such that 
		\begin{enumerate}
			\item $(u_1^{\eps_k},\ldots,u_N^{\eps_k})\overset{*}{\rightharpoonup}(u_1^\star,\dots, u_N^\star)$ in $L^\infty_{\rho_1\otimes\cdots\otimes\rho_N}$ and 
			\[\lim_{k\to \infty}\depsN(u_1^{\eps_k},\ldots,u_N^{\eps_k})=\sum_{i=1}^N\int u_i^\star d\rho_i=\max_{(u_1,\dots,u_N)\in\mathcal{B}'_N}\sum_{i=1}^N\int u_i d\rho_i,\]
			where 
			{\small	\[\mathcal{B}_N':=\{ (u_1,\dots,u_N)\in \bigtimes_{i=1}^N L^1_{\rho_i}, \ \sum_{i=1}^N u_i(x_i) \leq c(\xdots), \, \bigotimes_{i=1}^N\rho_i\text{-a.e.} (\xdots)\}.\]}
			\item For any sequence $(\gamma_{\eps_k})\subset \Pi(\rho_1.\dots, \rho_N)$ such that $\gamma_{\eps_k}\overset{\star}{\rightharpoonup}\gamma$, for some $\gamma\in\Pi(\rho_1.\dots, \rho_N) $,
			\[\liminf_{k\to\infty}\pepsN(\gamma_\eps)\ge \int \tilde c d\gamma,\]
			where $\tilde c$ is the $\bigotimes_{i=1}^N\rho_i$-essential regularization of $c$, as defined in \eqref{effetilda}. If for every $\eps>0$, we consider  the unique minimizer of \eqref{MM-EOT} $\bar\gamma_\eps$ (see Theorem \ref{th: maximizer MM}), then there exists $\gamma^\star$, such that, up to a subsequence $(\bar\gamma_{\eps_k})$, $\bar\gamma_{\eps_k}\overset{\star}{\rightharpoonup} \gamma^\star$ and \[\lim_{k\to\infty}\pepsN(\bar\gamma_{\eps_k})= \int \tilde c d\gamma^\star=\min_{\Pi(\rho_1.\dots, \rho_N)} \int \tilde c d\gamma.\]
			\item \[\max_{(u_1,\dots,u_N)\in\mathcal{B}'_N}\sum_{i=1}^N\int u_i d\rho_i=\min_{\Pi(\rho_1.\dots, \rho_N)} \int \tilde c d\gamma.\]
		\end{enumerate}
	\end{theorem}
	{\bf The multi-marginal Sinkhorn algorithm. }	Let $(u^0_1,\dots, u^0_N)$ be given by Corollary \ref{goodSequence}. For $n\ge 0$,  we define the sequence $(u^n_1,\dots,u^n_N)$, 
	\begin{equation}\label{eq:sequencesinkhornMM}
		\begin{cases}
			u^{n+1}_{1}:= \tranm{u_2^{n},\dots,u_N^{n}},\\
			u^{n+1}_{2}:= \tranm{u_1^{n+1},u_3^n\dots,u_N^{n}}\\
			\vdots \\
			u^{n+1}_{N}:= \tranm{u_1^{n+1}, \dots, u_{N-1}^{n+1}}.
		\end{cases}
	\end{equation}
	
	\begin{theorem}\label{thm:convergenceSinkhornMM}
		Let  $c:\Xtimes\to [0,+\infty]$ satisfy the assumption \eqref{ipotesi2MM}. Then the sequence \eqref{eq:sequencesinkhornMM} converges uniformly on compact sets to the unique maximizer  $(\bar u_1,\dots, \bar u_N)$ of $\deps^N$.
	\end{theorem}
	\begin{proof}
		The proof is the same as the one of Theorem \ref{thm:convergenceSinkhorn}. Indeed, Proposition \ref{prop:regolaritactransgen} can be extended in a natural way to the multi-marginal setting. The uniform bounds on the sequence come from Proposition \ref{betterpotentials}.
	\end{proof}
	
	\begin{remark}
		As well as Proposition \ref{prop:regolaritactransgen}, also Proposition \ref{prop:reularity1} and Proposition \ref{prop:lipschitz} apply to the multi-marginal case.
	\end{remark}

	\appendix
	
	\section{Properties of the dual functional $\deps$.}
	\begin{proposition}\label{prop:ineq1Dual}
		Let $\peps$ as defined in \eqref{hmkproblem} and $\deps$ as the one in \eqref{dualfunctional}, then
		\[\peps(\gamma)\ge \deps(u,v), \quad \text{for any }\gamma\in\Pi(\mu,\nu), u,v \in L^1_\mu\times L^1_\nu. \]
	\end{proposition}
	\begin{proof}
		Let $f:\R\to \R \cup \{+\infty\}$ be the convex function
		\[ f(\alpha)=\left\{ \begin{array}{ll}
			c\alpha +\eps \alpha\log\alpha - \eps\alpha & \mbox{if} \ 0\leq \alpha, \\
			+\infty & \mbox{otherwise.}
		\end{array} \right.\]
		The Legendre transform is 
		\[ f^*(\beta)= \eps e^{\frac{\beta-c}{\eps}}.\]
		Then the {\it generalised} Young inequality gives, for all $\alpha, \beta \in \R$, 
		\[ 	c\alpha +\eps \alpha\log\alpha - \eps\alpha + e^{\frac{\beta-c}{\eps}} \geq \alpha\beta. \]
		For $\gamma \in \Pi (\mu,\nu)$ we denote by $\gamma (x,y)$ the density of $\gamma$ with respect to 
		$\mu\otimes\nu$. If $u \in L^1_\mu$ and $\nu \in L^1_\nu$ we apply the Young inequality above with 
		$c=c(x,y)$, $\alpha=\gamma(x,y)$ and $\beta= u(x)+v(y)$ and, recalling that $\int d\gamma=1$, 
		\[\int c(x,y)d\gamma +\eps \mathcal{H}(\gamma | \mu\otimes\nu) \geq \int u(x) d\mu +\int v(y)  d\nu -\eps \int e^{\frac{u(x)+v(y)-c(x,y)}{\eps}} d\mu\otimes \nu +\eps.  \]
	\end{proof}
	
	Consider the set 
	\begin{equation}\label{Lexp}
		\LEnt_\mu :=\{u \in L^1_\mu \ : \ \int e^\frac{u(x)}{\eps}d\mu <+\infty\}
	\end{equation}
	then the following holds
	\begin{proposition}\label{domain}
		Let $u \in L_\mu^1$, $v\in L^1_\nu$. Then $\deps(u,v)>-\infty$ if and only if $u\in \LEnt_\mu$ and $v \in \LEnt_\nu$. 
	\end{proposition}
	\begin{proof} The fact that $\deps(u,v)>-\infty$ for every $u\in \LEnt_\mu$ and $v \in \LEnt_\nu$ is obvious by definition of $\LEnt$. 
		Let us assume $\deps(u,v)>-\infty$. Then 
		\[\int e^\frac{u(x)}{\eps}(\int e^\frac{v(y)-c(x,y)}{\eps} d\nu(y))d\mu(x)<+\infty.\]
		Now
		\[\int_Ye^\frac{v(y)-c(x,y)}{\eps}d\nu(y)\ge e^{\int_Y\frac{v(y)-c(x,y)}{\eps}d\nu(y)}>0,\]
		thanks to the integrability assumptions on $v$ and $c$ (see \eqref{ipotesi1}). Therefore
		\[\int_X e^\frac{u(x)}{\eps}d \mu(x)<+\infty.\]
	\end{proof}
	\begin{corollary}
		\[\sup_{(u,v)\in L^1_\mu \times L^2_\nu }\deps(u,v)=\sup_{(u,v)\in \LEnt_\mu \times \LEnt_\nu } \deps (u,v).\]
	\end{corollary}
	The functional $\deps:L^1_\mu \times L^1_\nu \to \R \cup \{-\infty\}$ is strictly concave and, for the previous inequality, bounded above. It follows that $\deps$ admits at most one maximizer. 
	Moreover, from the strict concavity it follows also that, if a couple $ (\bar u, \bar v)$ satisfies
	\[\deps (\bar u, \bar v) =\max_v \deps (\bar u, v), \]
	and 
	\[\deps (\bar u, \bar v) =\max_u \deps (u, \bar v), \]
	then it is the maximizing couple. 
	
	\begin{definition}\label{def:trasfentrop}
		Let $u\in \LEnt_\mu$. We define the $(c,\eps)$-transform 
		of $u$ as  
		\begin{equation}\label{eq:trasfentropmax}
			\tran{u}:= \underset{v\in L^1_\nu} \argmax \int_X u d\mu +\int_Y v d\nu -\eps \int_{X\times Y} e^\frac{u(x)+v(y)-c(x,y)}{\eps} d\mu\otimes \nu +\eps.
		\end{equation}
		In the same way, for every $v\in \LEnt_\nu$ we define the $(c,\eps)$-transform of $v$ as the function $v^{c,\eps}:X\to \R$ given by 
		\begin{equation}\label{eq:trasfentropmaxv}
			\tran{v}:=\underset{u\in L^1 _\mu}\argmax \int_X u d\mu +\int_Y v d\nu -\eps \int_{X\times Y} e^\frac{u(x)+v(y)-c(x,y)}{\eps} d\mu\otimes \nu +\eps.
		\end{equation}
	\end{definition}
	\begin{proposition}\label{prop:aboututransf}
		The function $\tran{u}$, defined in \eqref{eq:trasfentropmax} exists and is unique and 
		\begin{equation*}
			u^{c,\eps}(y):=-\eps\log\int_{X} e^{\frac{u(x)-c(x,y)}{\eps}}d\mu(x), \quad \text{for every} \ y\in Y.
		\end{equation*}
	\end{proposition}
	\begin{proof} Let $(u,v)\in L_\mu^{exp,\eps}(X\times L^{exp,\eps}_\nu(Y)$, then by Fubini's theorem
		\begin{align*}
			&\int_X u d\mu +\int_Y v d\nu -\eps \int_{X\times Y} e^\frac{u(x)+v(y)-c(x,y)}{\eps} d\mu\otimes \nu +\eps=\\&\int_Y\int_X u(x) d\mu(x)+v(y)-\eps\left(\int_X e^{\frac{u(x)-c(x,y)}{\eps}}d\mu(x) \right)e^{\frac{v(y)}{\eps}}d\nu(y)+\eps.
		\end{align*}
		Let us consider the strictly concave function $g(t)=a+t-\eps b e^{\frac t \eps}$. Since $\lim_{t\to\pm\infty}g(t)=-\infty$, it admits a unique maximum, which is the solution of 
		\begin{equation*}
			g'(t)=1-be^{\frac t \eps}=0,
		\end{equation*}
		i.e. 
		\begin{equation*}
			t=-\eps \log (b).
		\end{equation*}
		So, for every fixed $y$, the function $\int_X u(x) d\mu(x)+t-\eps\left(\int_X e^{\frac{u(x)-c(x,y)}{\eps}}d\mu(x) \right)e^{\frac{t}{\eps}}$ is maximized by 
		\begin{equation*}
			t=-\eps \log \int_{X} e^{\frac{u(x)-c(x,y)}{\eps}}d\mu(x)=:v(y).
		\end{equation*}
		We conclude by observing that $v\in L^{exp,\eps}_\nu (Y)$. Indeed  
		\begin{align*}
			\int _Y e^{\frac{v(y)}{\eps}}d\nu(y)&=\int _Y e^{-\log\int_{X} e^{\frac{u(x)-c(x,y)}{\eps}}d\mu(x)}d\nu(y)\\&=\int_Y \left(\int_{X} e^{u(x)-c(x,y)}d\mu(x)\right)^{-1}d \nu(y)<+\infty,
		\end{align*}
		because, by Jensen inequality 
		\begin{align*}
			\int_{X} e^{u(x)-c(x,y)}d\mu(x)\geq e^{\int_{X} {u(x)-c(x,y)}d\mu(x)},
		\end{align*}
	and by the fact that $u\in L^1(\mu)$ and by the assumption \eqref{ipotesi1},
	\[ \int_{X} {u(x)-c(x,y)}d\mu(x) \geq - \|u\|_{L^1_\mu} - \kcal .\]
	\end{proof}
	
	\section{A semicontinuity theorem}
	\begin{theorem}\label{semicontinuity}
		Let $(\Omega, \Sigma, \lambda)$ be a measure space with $\lambda$ $\sigma$-finite on $\Omega$ and let $p\in (1,+\infty)$. 
		Let 
		\[f:\Omega \times \R^d \to \overline{R},\]
		be $\Sigma \times \mathcal B(\R^d)$-measurable and such that 
		\begin{enumerate}
			\item[(i)] For $\lambda$-a.e. $\omega \in \Omega$, $\xi \to f(\omega, \xi)$ is convex and lower semi-continuous;
			\item[(ii)]	for a function $a \in L^1_\lambda$ and a number $b\ge0$  
			\[f(\omega, \xi) \geq -a(\omega)-b|\xi|^p.\] 
		\end{enumerate}	
		Then 
		\[ u\mapsto \int _\Omega f(\omega,u(\omega))d \lambda, \] 
		is lower semi-continuous in $L^p_\lambda (\Omega,\R^d)$ with respect to the weak topology. 
		
	\end{theorem}
	
	\section{Existence of a Vitali covering for separable metric spaces}
	Let $(X,d)$  be a metric space. Let $\lambda$ be a Borel  measure such that $\lambda$ is finite on bounded sets.
	\begin{definition}
		We say that a family of sets $\mathcal{S}$ is fine at a point $x\in X$, if 
		\[\inf \{\mathrm{diam}(S) \, : \, S\in \mathcal S \ \text{and} \ x\in S\}=0\]
	\end{definition}
	
	\begin{definition}[see Definition 2.8.16 in \cite{FedSpringer1996}]\label{def:vitalicovering}
		We say that a family of sets $\mathcal{V}$ is a $\lambda-$Vitali covering if it satisfies the following properties:
		\begin{enumerate}
			\item $\mathcal{V}$ is a family of Borel sets;
			\item $\mathcal{V}$ is fine at each point $x\in X$;
			\item the following condition holds: if $\mathcal{S}\subset \mathcal V$, $Z\subset X$, and $\mathcal{S}$ is fine at each point $x\in Z$, then $\mathcal S$ has a countable disjoint sub-family which covers $\lambda $-almost all of $Z$.
		\end{enumerate}
	\end{definition}
	\begin{definition}
		Let $f:X\to \R\cup\{+\infty\}$  a $\lambda$- measurable function. Whenever we have a family of Borel sets $S$ which is fine at $x\in X$ we can define 
		\begin{multline}
			\mathcal V-\limsup\limits_{S\to x}\frac{1}{\lambda(S)}\int_S f  d\lambda:=\\ \limsup\limits_{\rho\to 0}\left\{\frac{1}{\lambda(S)}\int_S f  d\lambda \,: \, S\in\mathcal V, x\in S \ \text{and} \ \mathrm{diam}(S)<\rho\right\}.
		\end{multline}
		Similarly we define $\mathcal V-\liminf\limits_{S\to x}\int_S f  d\lambda$ and  $\mathcal V-\lim\limits_{S\to x}\int_S f  d\lambda$.
	\end{definition}
	\begin{definition}\label{def:lexicographic}
		Given $a:=(a_n)_{n\in \N}$ and $b:=(b_n)_{n\in\N}$, such that $a_n,b_n\in\N$ for every $n$, we say that $a<b$ with respect to the lexicographic order if there exists $k\ge 0$, such that $a_i=b_i$ for every $0\le i \le k-1$ and $a_{k}<b_k$.
	\end{definition}
	\begin{theorem}[see Theorem 2.9.8 in \cite{FedSpringer1996}] \label{thm:vitalicovering}
		Let $(X,d)$ a metric space and $\mathcal V$ a Vitali covering.	Let $f:X\to \R\cup\{+\infty\}$  a $\lambda$- measurable function such that \[\int_A |f|d\lambda<\infty, \] 
		for every bounded $\lambda$-measurable set $A$, then 
		\[\mathcal V-\lim_{S\to x}\frac{1}{\lambda(S)}\int_S f  d\lambda = f(x), \quad \text{for}\ \lambda \text{- a.e.} \ x \in X.\]
	In other words, for $\lambda$-a.e. $x\in X$, $x$ is a generalized Lebesgue point.
	\end{theorem}
	\begin{proposition}\label{prop:existenceVitali}
		Let $(X,d)$ be a separable metric space and let  $0<r<\frac{1}{3}$. Then there exists a countable family $\{Q_{k,i_k} \, : \,  k\in \N, i_k\in N_k\subset \N_{>0}  \}$ satisfying the following properties:
		\begin{enumerate}
			\item $X=\bigcup_{i_k\in N_k} Q_{k,i_k} $, for every $k\in\N$;
			\item  for $k\ge m$, $Q_{k,i_k}\cap Q_{m,i_m}=\emptyset$ or $Q_{k,i_k}\subset  Q_{m,i_m}$ ;
			\item for every $k\in\N$ and $i_k\in N_k$ there exists a point $x_{k,i_k}$ such that 
			\[\Qkik \subset \overline {B(\xkik, Cr^k)},\]
			where $C>0$ is a  constant depending only on $r$\
			
		\end{enumerate}
		In particular, given any Borel measure $\lambda$ on $X$, the family $\{Q_{k,i_k} \, : \,  k\in \N, i_k\in N_k\subset \N_{>0}  \}$ is a $\lambda$-Vitali covering in the sense of Definition \ref{def:vitalicovering}.
	\end{proposition}
	\begin{proof} Fix a point $x_{0,1}$ and define a set of points 
		\[A_0:=\{ x_{0,i_0} \, : \, i_0\in N_0\subset \N_{>0}  \}\]
		such that $x_{0,1}\in A_0$, $d(x_{0,i}, x_{0,j})\ge 1$ for every $i,j\in N_0$, $i\neq j$ and  $A_0$ is maximal. 
		We observe that the separability of $X$ ensures that the set $A_0$ is countable. Then, for each $k\in \N$ we define $A_k:=\{ \xkik \, : \, i_k\in N_k\subset \N_{>0}  \}$ such that $A_{k-1}\subset A_k$, $d(x_{k,i}, x_{k,j})\ge r^k$ for every $i,j\in N_k$, $i\neq j$ and  $A_k$ is maximal.\\
		\textbf{Step 1:} We, now, assign to every $\xkik\in A_k\setminus A_{k-1}$ a unique parent  $x_{k-1,i_{k-1}}\in A_{k-1}$. We start with the point $x_{1,i_1}\in A_1\setminus A_0$. 
		Let $0<r<\eps<2r$. We define the set \[G_{0,i_1}:= \{ x_{0,j} \, : \,  d(\xoneione, x_{0,j})<1+\eps  \}.  \]Notice that $G_{0,i_1}$ is not empty thanks to the maximality of $A_0$.
		The parent of $x_{1,i_1}$ will be the unique point $x_{0,i_0}$, such that \[i_0:=\min \{j\in N_0 \, : \,  x_{0,j}\in G_{0,i_1}\}.\]
		Then we assign to each point a \textbf{word}, i.e. an infinite sequence of natural numbers in the following way. If $x_{0,i_0}\in A_0$ we relabel it $x_{(i_0,0,\dots)}$, i.e. $i_l=0$ for every $l\ge 1$. If $\xoneione\in A_1\setminus A_0$ we relabel it $x_{(i_0,i_1,0,\dots)}$, where $i_0$ is the index of the parent of $\xoneione$ and $i_l=0$ for every $l\ge 2$. For any $k\in \N$, we consider $x_{k+1,i_{k+1}}\in A_{k+1}\setminus A_{k}$ and the set 
		\[G_{k,i_{k+1}}:= \{ x_{k,j}\in A_k \, : \,  d(x_{k,j}, x_{k+1,i_{k+1}})<r^k+\eps^{k+1}  \}.  \] The parent of $x_{k+1,i_{k+1}}$ is the point $\xkik\in G_k$ with label $(i_0,i_1,\dots,i_{k},0,\dots)$, which has the smallest label with respect to the lexicographic order (see Definition \ref{def:lexicographic}), among all the labels of the points $x_{k,j}\in G_{k,i_{k+1}}$. We observe that by construction 
		\begin{itemize}
			\item if $ x_{k,i_k}\in A_k$, the label of $\xkik$ will be $(i_0,\dots, i_k,0,\dots)$, with $i_j\ge 0$ for every $0\le j\le k$ and $i_j=0$, for every $j>k$;
			\item to each $\xkik$ there exist unique $i_0,\dots, i_{k}\in\N^{k}$ such that $\xkik=x_{(i_0,\dots, i_k,0,\dots)}$. \textbf{Notation:} for this reason a point can be denoted either as $\xkik$ or $x_{(i_0,i_1,\dots, i_k,0,\dots)}$, according to the necessity. 
		\end{itemize}
		
		\textbf{Step 2:} Define the sets $Q_{k,i_k}$. Starting from $k=0$, for $i_0\in N_0$, we define 
		\begin{align*}
			&Q_{0,i_0}:=Q_{(i_0,0\dots)}:=\\&\overline{\{ x_{m,i_m}=x_{(i_0,j_1\dots, j_m, 0, \dots)} \, : \, m>0 \}}\setminus \bigcup_{1\le j_0<i_0}Q_{(j_0,0\dots)},
		\end{align*}
		where $\{ x_{m,i_m}=x_{(i_0,j_1\dots,j_m, 0, \dots)} \, : \, m>0 \}$ is the set of \textbf{children} of $x_{0,i_0}$.
		For any $k$ and for any point $x_{k,i_k}=x_{(i_0,\dots,i_k,0\dots)}$ we define 	$\Qkik=Q_{(i_0,\dots,i_k,0\dots)}$ as 
		{\small \begin{align*}
				&Q_{(i_0,\dots,i_k,0\dots)}:= Q_{(i_0,\dots,i_{k-1},0\dots)}\cap
				\\&\overline{\{ x_{m,i_m}=x_{(i_0,\dots,i_k,j_{k+1}, \dots, j_m, 0, \dots)} \, : \, m>k \}}\setminus \bigcup_{(j_0,\dots,j_k,0\dots)<(i_0,\dots,i_k,0\dots)}Q_{(j_0,\dots,j_k,0\dots)},
		\end{align*}}where we have intersected we the set generated by the \textbf{parent} of $\xkik$ and where $(j_0,\dots,j_k,0\dots)<(i_0,\dots,i_k,0\dots)$ has to be meant with respect to the lexicographic order.

		\textbf{Step 3:} We prove Properties 1.-3. 
		Property 2. follows directly by the definition of the sets $\Qkik$. Concerning Property 1, we fix a level $k\in \N$ and  we observe that \footnote{Since at each step $k$, every $\Qkik$ is defined by taking the intersection with the $Q_{k-1,i_{k-1}}$ generated by the parent, to be rigorous, one should first prove that $X=\bigcup_{i_0\in N_0}Q_{0,i_0}$ and then proceed by induction.} 
		\begin{equation*}\label{eq:equalityunion}
			\bigcup_{i_k\in N_k} Q_{k,i_k}=\bigcup_{i_k\in N_k}\overline{\{ x_{m,i_m}=x_{(i_0,\dots,i_k,j_{k+1}, \dots, j_m, 0, \dots)} \, : \, m>k \}}.
		\end{equation*}
		Let $x\in X$, define a sequence $\{x_{l,i_l}\}$ such that 
		\begin{itemize}
			\item $x_{l,i_l}\in A_l$, for every $\in\N$;
			\item $d(x,x_{l,i_l})<r^l$.
		\end{itemize}
		Then we have that 
		\[d(x_{l,i_l},x_{l+1,i_{l+1}})<d(x_{l,i_l},x)+d(x,x_{l+1,i_{l+1}})<r^l+r^{l+1}<r^l+\eps ^ {l+1},\]
		which means that $x_{l,i_l}\in G_{l,i_{l+1}} $ and thus it is a competitor to be the parent of $x_{l+1,i_{l+1}}$. Therefore the label $(i_0,\dots,i_l,0,\dots)$ of $x_{l,i_l}$ is bigger or equal, with respect to the lexicographic order, to the label of  $x_{l+1,i_{l+1}}$. In particular if we consider $x_{0,i_0}$, the first point of the sequence, for every $x_{l,i_l}$ , the first index of the label will be at most $i_0$. This implies, since $i_0\in \N$, that there exists a subsequence of $\{x_{l,i_l}\}$ such that all the points have the same first index in the label. Continuing inductively we obtain a subsequence such that the first $k$ indexes of the label are fixed. If the first letters are, let's say $i_0,\dots, i_k$ this means that our subsequence of $\{x_{l,i_l}\}$ is included in the set $\{x_{(i_0,\dots,i_k,j_{k+1}, \dots, j_m, 0, \dots)} \, : \, m>k \}$.

		We now prove Property 3. We observe that if $$x_{m,i_m}\in  \{ x_{m,j_m}=x_{(i_0,\dots,i_k,j_{k+1}, \dots, j_m, 0, \dots)} \, : \, m>k \},$$ that is $x_{m,i_m}$ is a child of $\xkik$, then 
		\[d(\xkik,x_{m,i_m})<\sum_{j=k}^m r^j + \eps ^{j+1} <3\sum_{j=k}^m r^j <3\frac{1}{1-r}r^k. \]
		This implies that 
		\[Q_{k,i_k}\subset\overline{B(\xkik,Cr^k)},\]
		with $C\ge\frac{3}{1-r}$.
		
		\textbf{Step 4:} We conclude by showing that the family is a Vitali covering. By construction $Q_{k,i_k}$ are Borel sets  and $\mathcal{V}:=\{Q_{k,i_k}\}$ is fine at each $x\in X$, by Property 3. Let $\mathcal S \subset \mathcal V$ and $Z\subset X$, such that $\mathcal S$ is fine at every $z\in Z$. We construct a countable disjoint sub-family $\mathcal S'$, such that $\bigcup_{S\in \mathcal S'}S=Z$.
		One can  write $\mathcal{S}:=\{Q_{k,i_k} \, : \, k\in \N, i_k\in M_k\subset N_k  \}$, where $M_k$ can also be empty. 
		Let us start from $k=0$ and consider the family $\{Q_0,i_0\}_{i_0\in M_0}$. Then $Q_{0,i}\cap Q_{0,j}=\emptyset $, for every $i\neq j \in M_0$. Then we consider $\{Q_{1,i_1}\}_{i_1\in M_1}$ and we discard all the $Q_{1,i_1}$s such that $Q_{1,i_1}\subset Q_{0,i_0}$, for some $i_0$. We do the same for every $k$. Then we get $\mathcal S':=\{Q_{k,i_k} \, : \, k\in \N, i_k\in M'_k\subset M_k  \}$ , where $M'_k$ is the set of indexes at level $k$ which were not discarded. 
	\end{proof}
	
	\section{$\lambda$-essential relaxation}
	Let $Z$ be a Polish space and let $\lambda	\in \Prob(Z)$ . 
	Let $f: Z \to [0,+\infty] $ be a lower semi-continuous function. 
Define 
	\begin{equation}\label{effetilda}
		\tilde f (z):= \sup \{t \, : \, \, \exists r>0, \, \text{s.t.} \, \lambda (B_r(z)\cap \{f>t\} )= \lambda(B_r(z)) \}.
	\end{equation}
%\begin{equation}\label{effetilda}
%	\tilde f (z):= \sup \{t \, : \, \forall \eps>0, \, \exists r>0, \, \text{s.t.} \, \lambda (B_r(z)\cap \{ t-\eps <f\} )= \lambda(B_r(z)) \}.
%	\end{equation}
	\begin{theorem} \label{propertiestilde}  The function $\tilde f$ satisfies the following: 
		\begin{enumerate}
			\item $\tilde{f}(z)= +\infty$ if $z \not \in \spt \lambda$;
			\item $f(z)\leq \tilde f(z)$;
			\item if $f$ is $\lambda$-integrable, then $f(z)=\tilde f (z)$, \ $\lambda-a.e.$;
			\item $\tilde f$ is l.s.c. .
		\end{enumerate} 
	\end{theorem}
	\begin{proof} The first point follows directly from the definition since any $t\in \R$ fulfils the  requirement. 
		Denote by 
\[S_f (z):=  \{t \, : \, \exists r>0, \, \text{s.t.} \, \lambda (B_r(z)\cap \{ f>t\} )= \lambda(B_r(z)) \}.\]
%	\[S_f (z):=  \{t \, : \, \forall \eps>0, \, \exists r>0, \, \text{s.t.} \, \lambda (B_r(z)\cap \{ f>t-\eps\} )= \lambda(B_r(z)) \}.\]
	If $f$ is lower semi-continuous, there exists $r>0$ such that $B(z,r) \subset \{ f>f(z)\}$, which means that $f(z) \in S_f (z)$ and thus $f(z) \leq \tilde{f} (z)$ .
%		If $f$ is lower semi-continuous, then $\forall \eps>0$ there exists $r>0$ such that $B(z,r) \subset \{ f(z)-\eps <f\}$ then $f(z) \in S_f (z)$ so that $f(z) \leq \tilde{f} (z)$ . 
		
		Let $z_0$ be a point such that $f(z_0)<\tilde f(z_0)$. We show that $z_0$ cannot be a generalized Lebesgue point of $f$.  Thanks to Theorem \ref{thm:vitalicovering}, this conclude the proof of 3. 
%		Let $\varepsilon>0$ be such that
%		\[f(z_0)<\tilde f(z_0)-\varepsilon<\tilde f(z_0).\]
%		
%		By the definition of $\tilde f$  there exists $r>0$ such that the  ball $B(z_0,r)$ is $\lambda$-essentially contained in 
%		$\{\tilde f(z_0)-\varepsilon<f\}$. 

	By the definition of $\tilde f$  there exists $r>0$ such that the  ball $B(z_0,r)$ is $\lambda$-essentially contained in 
	$\{f>\tilde f(z_0)\}$. 

		Let $\mathcal{V}$ be a Vitali covering of $Z$, which exists by Proposition \ref{prop:existenceVitali}.
		Set $\delta_0:=\frac r2$. Now for all $S\in\mathcal{V}$ such that $z_0\in S$ and $diam(S)<\delta_0$ we have 
		$S\subset B(z_0,r)$ and thus every such $S$ is $\lambda$-essentially contained in the set $\{f>\tilde f(z_0)\}$.
%		is $\lambda$-essentially contained in the set
%		$\{\tilde f(z_0)-\eps<f\}
		Now
		\begin{align*}
			&\mathcal{V}- \lim_{S\to z_0 }\int_S f(z) d\lambda\\
			&=\lim_{\rho\to 0}\left\{\frac{1}{\lambda(S)}\int_S f (z) d\lambda \, : \, S \in\mathcal V, z_0\in S \ \text{and} \ \mathrm{diam}(S)<\rho \right\}\\
			&=\lim_{\rho\to 0, \rho<\delta_0}\left\{\frac{1}{\lambda(S)}\int_S f (z) d\lambda \, : \, S \in\mathcal V, z_0\in S \ \text{and} \ \mathrm{diam}(S)<\rho \right\}\\
			&>f(z_0).
		\end{align*}
		
%			\begin{align*}
%			&\mathcal{V}- \lim_{S\to z_0 }\int_S f(z) d\lambda\\
%			&=\lim_{\rho\to 0}\left\{\frac{1}{\lambda(S)}\int_S f (z) d\lambda \, : \, S \in\mathcal V, z_0\in S \ \text{and} \ \mathrm{diam}(S)<\rho \right\}\\
%			&=\lim_{\rho\to 0, \rho<\delta_0}\left\{\frac{1}{\lambda(S)}\int_S f (z) d\lambda \, : \, S \in\mathcal V, z_0\in S \ \text{and} \ \mathrm{diam}(S)<\rho \right\}\\
%			&>\tilde f(z_0)-\varepsilon >f(z_0).
%		\end{align*}

		Let $z_n \to z$ we aim to prove that $\tilde{f} (z) \leq \liminf_{n\to \infty} \tilde{f} (z_n) .$
		Let $\alpha>0$ and let $r_\alpha$ be such that $B(z, r_\alpha)$ is $\lambda$-essentially contained in 
		$\{\tilde{f}(z) -\alpha< f \}$. Since $z_n \to z$ for $n$ big enough $B(z_n,\frac{r_\alpha}{2}) \subset B(z, r_\alpha)$ and this, by definition, implies that $\tilde{f}(z_n) > \tilde{f}(z) -\alpha$. Since $\alpha$ is arbitrary we conclude.

	\end{proof}
	%%%%%%%%%%%%%%%%%%%%%%%%%%%%%%%%%%%%%%%%%%%%%%%%%%%%%%%%%%%%
	\noindent{{\bf{Acknowledgments:}} 
		C.B. was partially funded by \emph{Deutsche Forschungsgemeinschaft (DFG -- German Research Foundation) -- Project-ID 195170736 -- TRR109}.
		
		C.B. and L.D.P. acknowledge the support of GNAMPA-INDAM. 
		
		The work of L.D.P. is partially financed by the {\it``Fondi di ricerca di ateneo''} of the University of Firenze and partially financed by the EU-Next Generation EU, (Missione 4, Componente 2, Investimento 1.1 {\it Progetti di Ricerca di Rilevante Interesse Nazionale} (PRIN), CUPB53D23009310006 - (2022J4FYNJ), \emph{Variational methods for stationary and evolution problems with singularities and interfaces} .
		
			A.K. is currently funded by a scholarship of the Dipartimento di Matematica e informatica of the Università di Firenze. 	
	Finally, the authors thank Tapio Rajala for the generous help with Proposition \ref{prop:existenceVitali}.

		%%%%%%%%%%%%%%%%%%%%%%%%%%%%%%%%% BIBLIOGRAPHY %%%%%%%%%%%%%%%%%%%%%%%%%%%%%%%%%%%%%%%%%%%%%%%%%%%%%%

		%\bibliography{MyBiblio-LDP}
		
		%\bibliographystyle{plain} 

	\end{document}